\documentclass[11pt,a4paper]{amsart}

\usepackage{}
\usepackage{amsthm}
\usepackage{amsfonts}
\usepackage{amsmath}
\usepackage{mathrsfs}
\usepackage{amssymb}
\usepackage{amsfonts}
\usepackage{amsbsy}

\usepackage{CJK,CJKnumb}
\usepackage{color}              %
\usepackage{indentfirst}        %
\usepackage{latexsym,bm}        %
\usepackage{graphicx}
\usepackage{cases}
\usepackage{pifont}
\usepackage{txfonts}
\usepackage{xcolor}
\usepackage[all,knot,poly]{xy}
\usepackage{ifpdf}
\ifpdf
  \usepackage[colorlinks,final,backref=page,hyperindex]{hyperref}
\else
  \usepackage[colorlinks,final,backref=page,hyperindex,hypertex]{hyperref}
\fi

\allowdisplaybreaks[4]  
\CJKtilde

\numberwithin{equation}{section}

\newcommand{\al}{\alpha}
\newcommand{\be}{\beta}
\newcommand{\de}{\delta}
\newcommand{\De}{\Delta}
\newcommand{\ep}{\epsilon}
\newcommand{\ga}{\gamma}
\newcommand{\Ga}{\Gamma}
\newcommand{\ka}{\kappa}
\newcommand{\la}{\lambda}

\newcommand{\La}{\Lambda}
\newcommand{\mi}{\mbox{\upshape Id}}
\newcommand{\mh}{\mathcal {H}}
\newcommand{\mk}{\mathcal {K}}
\newcommand{\mx}{\mathcal {X}}
\newcommand{\om}{\omega}
\newcommand{\ot}{\otimes}
\newcommand{\si}{\sigma}
\newcommand{\te}{\theta}
\newcommand{\ve}{\varepsilonup}
\newcommand{\vs}{\varsigmaup}
\newcommand{\ze}{\zeta}
\newcommand{\ds}[1]{{\displaystyle{#1}}}
\newcommand{\nb}[1]{{\bfseries #1}}
\newcommand{\mb}[1]{\mbox{\bfseries #1}}

\newcommand{\ms}[1]{\mbox{\sffamily #1}}
\newcommand{\bs}[1]{{\scriptsize\mbox{#1}}}
\newcommand{\lan}{\langle}
\newcommand{\ran}{\rangle}
\newcommand{\lb}{\left(}
\newcommand{\rb}{\right)}

\begin{document}

\newtheorem{theorem}{Theorem}[section]

\newtheorem{lemma}[theorem]{Lemma}

\newtheorem{corollary}[theorem]{Corollary}
\newtheorem{proposition}[theorem]{Proposition}

\theoremstyle{remark}
\newtheorem{remark}[theorem]{Remark}

\theoremstyle{definition}
\newtheorem{definition}[theorem]{Definition}

\newtheorem{example}[theorem]{Example}

\title[Drinfel'd doubles of the $n$-rank Taft algebras and knot invariants]{Drinfel'd doubles of the $n$-rank Taft algebras and a generalization of the Jones polynomial}

\author[Feng, Hu and Li]{Ge Feng, \quad Naihong Hu$^\star$ \quad and \quad Yunnan Li$^*$}
\address{School of Mathematical Sciences, Shanghai Key Laboratory of PMMP, East China Normal University,
	Shanghai 200241, China} \email{52170601010@stu.ecnu.edu.cn}
\address{School of Mathematical Sciences, Shanghai Key Laboratory of PMMP, East China Normal University,
	Shanghai 200241, China} \email{nhhu@math.ecnu.edu.cn}
\address{School of Mathematics and Information Science, Guangzhou University, Waihuan Road West 230, Guangzhou 510006, China} \email{ynli@gzhu.edu.cn}

\thanks{$^\star$ Supported in part by the NNSF of China (Grant Nos. 11771142, 11801394) and in part by the Science and Technology Commission of Shanghai Municipality (No. 18dz2271000).
}
\thanks{$^*$ Corresponding author, supported in part by the NNSF of China (Grant No. 12071094).}




\begin{abstract}
In the paper, we describe the Drinfel'd double structure of the $n$-rank Taft algebra and all of its simple modules, and then endow its $R$-matrices with an application to knot invariants. The knot invariant we get is a generalization of the Jones polynomial, in particular, it coincides with the Jones polynomial in rank $1$ case, while in rank $2$ case, it is the one-parameter specialization of the two-parameter unframed Dubrovnik polynomial, and in higher rank case it is the composite ($n$-power) of the Jones polynomial.
\end{abstract}

\maketitle

\section{Introduction}

Knots and $3$-manifolds are important research topics in low dimensional topology. In the early period, topologists found some classical knot invariants using some powerful tool, like homotopy groups or by means of elementary
algebraic combinatorics, like Alexander polynomials, later on, Jones,
Kauffman, HOMFLY polynomials appeared successively.
In the late 1980s, quantum invariants were discovered in the original work \cite{RT1,RT2}  and \cite{Tu,Wi}, where quantum groups and the relevant representations are used to construct knot invariants
and even invariants of $3$-manifolds. Many classical invariants of knots were recovered via the approach of quantum invariants. With the studies of quantum groups and their representation theories towards to categorifications,
the studies of knot theory have stepped into a new age (see \cite{Tu1, Ka, Oht, Ro, Kh, CR, LQR, WW, Tub}, etc.) and some important progress, for example, of the Labastida-Marino-Ooguri-Vafa conjecture in string theory, see \cite{LP}, and the volume conjecture in geometry, see \cite{CL}, \cite{CY} etc. and references therein.

In the paper, we adopt Kauffman-Radford's construction of the oriented regular
invariants arising from the oriented quantum algebras in \cite{KR}, which strongly depends on the
quasitrangular structure of quantum groups. The main source of
quasitriangular Hopf algebras is the Drinfel'd double,
proposed by Drinfel'd \cite{Dri}.
The oriented quantum algebras we used here are the Drinfel'd doubles of a certain kind of finite dimensional Hopf algebras, {\it the $n$-rank Taft algebras} we defined.
In 1971,  Taft \cite{Ta} found a famous non-semisimple Hopf algebra of prime square dimension.
That is the so-christened Taft algebra.  In 2001, Hu \cite{Hu2} defined the quantized universal enveloping algebras of abelian Lie algebras.
In particular, when $q$ is a root of unity, there exists a finite dimensional quotient Hopf algebra, called {\it the $n$-rank Taft algebra}. In rank $1$ case, it is the Taft algebra.

The paper is organized as follows. We first determine the Drinfel'd double $D(\mathscr{\bar A}_q(n))$ of the $n$-rank Taft algebra in Section 3 and give the sufficient and necessary condition for $D(\mathscr{\bar A}_q(n))$ to be a ribbon Hopf algebra with its ribbon element $\nu$ (Theorem \ref{ribb}). Also, we point out that
$\mathscr{\bar A}_q(n)$ is a Hopf 2-cocycle twist of the $n$-fold tensor product $\mathscr{\bar A}_q(1)^{\ot n}$ (we say the {\it rank $n$} Taft algebra) (Theorem \ref{twist}), thus they have equivalent Yetter-Drinfel'd module categories.
In Section 4, we recall the oriented quantum algebra and Kauffman-Radford's approach to regular isotopic invariants of knots and links, and improve the Kauffman-Radford construction to yield the ambient isotopic invariants, the oriented ambient invariants associated with this kind of ribbon Hopf algebras (Proposition \ref{wnor}).
Section 5 concentrates on exploring knot ambient invariants associated with the Drinfel'd double $D(\mathscr{\bar A}_q(n))$.

In Section 5, we first determine all simple $D(\mathscr{\bar A}_q(n))$-modules (Theorem \ref{dmod}), and give the sufficient and necessary condition for a simple $D(\mathscr{\bar A}_q(n))$-module to be self-dual (Theorem \ref{sdual}).
Our aim is to obtain a kind of quantum knot invariants subordinating to this Drinfel'd double, which is a generalization of the Jones polynomial. In rank $1$ case, we recover the Jones polynomial (Example \ref{ex1}). For the $2$-rank Taft algebra, we work with the $4$-dimensional self-dual simple $D(\mathscr{\bar A}_q(2))$-module and obtain a new semisimple $16\times 16$ braiding $R$-matrix with the minimal polynomial of degree $3$ (Example \ref{ex2}), from which we derive the same skein relation of Dubrovnik type. To our surprise, it is not related to the quantum groups of $BCD$ type, but the latter is! We thus investigate the endomorphism algebra $\ms{End}_{D(\bar{\mathscr{A}}_q(2))}\lb\bar{\mathscr{A}}_{(\ell-1)\ep_2,i_0(\ep_1+\ep_2)}^{\ot r}\rb$ of $D(\mathscr{\bar A}_q(n))$-module $\bar{\mathscr{A}}_{(\ell-1)\ep_2,i_0(\ep_1+\ep_2)}^{\ot r}$ with the dimension-square of the Temperley-Lieb algebra, which is different than the Birman-Wenzl-Murakami algebra (arising from the quantum groups of $BCD$ type) but contains its quotient as the proper subalgebra (Theorem \ref{end}, Proposition \ref{subalg}). These distinguished features stimulate us to compute some knot invariants for those knots or links with lower crossings following Kauffman-Radford's algorithm. The result indicates the square of the Jones polynomial in the rank $2$ case (Theorem \ref{sqj}). We argue it is true in generality in terms of the skein relations. As a generalization to higher rank, we can derive that the ambient isotopic invariant, associated to the double of $n$-rank Taft algebra, is the
composite ($n$-power) of the Jones polynomial (Corollary \ref{nrank}).

\section{Notation and Definitions}

Throughout the paper, let $k$ be an algebraically closed field of
characteristic $0$. Due to Manin \cite{Ma}, the
{\itshape quantum affine $n$-space} over $k$ is a quadratic algebra
$$
k[A_q^{n|0}]=k\bigl\{x_1,\ldots,x_n\bigr\}/(x_ix_j-qx_jx_i, \ i>j\,).
$$
Now let $\La=\bigoplus\limits_{i=1}^n\mathbb{Z}\ep_i$ be a free abelian group of $n$-rank, where $\ep_i=(\de_{1i},\ldots,\de_{ni})$, and $\de_{ij}$ is the Kronecker symbol.
Hence $\{\ep_i\}_{1\leq i\leq
n}$ is a  basis of $\La$. Besides, for $\al=(\al_1,\ldots,\al_n)$, $\be=(\be_1,\ldots,\be_n)\in\La$,
we denote $\de_{\al,\be}=\de_{\al_1,\be_1}\cdots\de_{\al_n,\be_n}$.

As in \cite{Hu1}, a skew bicharacter
$\te:\La\times\La\rightarrow k^\times$ on $\La$ is defined as follows,
\[\te(\al,\be)=q^{\al*\be-\be*\al}, \quad \al, \ \be\in\La, \quad q\in
k^\times,\]
where $\al*\be=\sum\limits_{j=1}^{n-1}\sum\limits_{i>j}\al_i\be_j$.
In particular,
\begin{equation}\te(\ep_i,\ep_j)=
\begin{cases}q, &i>j\,,
\\1, &i=j\,,
\\q^{-1}, &i<j
\end{cases}
\end{equation}
and
\[\te(\al,\be)^{-1}=\te(\be,\al)=\te(-\al,\be)=\te(\al,-\be),\quad\te(0,\al)=\te(\al,0)=\te(\al,\al)=1.\]

Now let $\La_+\coloneqq \bigl\{\al=
\sum\limits_{i=1}^n\al_i\ep_i\in\La \mid \al_i\in\mathbb{Z}^{\geq 0}\bigr\}$,
and $x^\al=x_1^{\al_1}x_2^{\al_2}\cdots x_n^{\al_n}$
be any non-zero monomial in $k[A_q^{n|0}]$,
then $\bigl\{x^\al\mid \al\in\La_+\bigr\}$ constitutes
a canonical basis of $k[A_q^{n|0}]$
. Moreover,
 $x^\al x^\be=q^{\al*\be}x^{\al+\be}=\te(\al,\be)x^\be
x^\al$, by definition.

For any $\al\in\La$, consider the algebra automorphism $K(\al)$ of $k[A_q^{n|0}]$ as follow,
\[K(\al)(x^\ga)=\te(\al,\ga)q^{\lan\al,\ga\ran}
x^\ga, \qquad x^\ga\in k[A_q^{n|0}],\]
where $\lan\cdot,\cdot\ran:\La\times\La\rightarrow
\mathbb{Z}$ is a symmetric bilinear form with $\lan\ep_i,\ep_j\ran=\de_{ij}$.

Apparently, $K(\al)K(\be)=K(\al+\be)=K(\be)K(\al)$. When $q=1$,
 $K(\al)=\ms{id}$.

As in \cite{Hu2}, let $\mathscr{A}_q(n)$ be an associative $k$-algebra with $1$,
 generated by the symbol $K(\ep_i)^{\pm 1},x_i~(1\leq i\leq n)$,
and associated with a skew bicharacter $\te$. It satisfies the relations
\[
\begin{array}{ll}
\ms{(r1)}~&K(\ep_i)K(\ep_j)=K(\ep_i+\ep_j)=K(\ep_j)K(\ep_i), \quad
K(\ep_i)^\pm K(\ep_i)^\mp=K(0)=1,\\
\ms{(r2)}~&K(\ep_i)^\ell=K(\ell\ep_i)=1, \quad \bigl(\,\ell=char(q)\,\bigr)\\
\ms{(r3)}~&K(\ep_i)x_jK(\ep_i)^{-1}=\te(\ep_i,\ep_j)q^{\de_{ij}}x_j,\\
\ms{(r4)}~&x_ix_j=\te(\ep_i,\ep_j)x_jx_i,
\end{array}
\]
where $\ms{char}(q)=\ms{min}\bigl\{r\in\mathbb{Z}^{>0}\mid q^r=1\bigr\}$. We view $\mathscr{A}_q(n)$ as the {\itshape quantized universal enveloping algebra of an abelian Lie algebra of dimension $n$}.
Obviously,
$\mathscr{A}_q(n)$ contains the quantum affine $n$-space $k[A_q^{n|0}]$ as its subalgebra.
We have the Hopf algebra structure of $(\mathscr{A}_q(n),\De,\ve,s)$ as follows
\[
\begin{array}{l}
\De: \mathscr{A}_q(n)\rightarrow\mathscr{A}_q(n)\ot\mathscr{A}_q(n), \quad
\De(K(\ep_i)^\pm)=K(\ep_i)^\pm\ot
K(\ep_i)^\pm, \quad \De(x_i)=x_i\ot1+K(\ep_i)\ot
x_i,\\
\ve:\mathscr{A}_q(n)\rightarrow
k, \quad \ve(K(\ep_i)^\pm)=1,\quad  \ve(x_i)=0,\\
s:\mathscr{A}_q(n)\rightarrow\mathscr{A}_q(n),\quad s(K(\ep_i)^\pm)=K(\ep_i)^\mp,\quad
s(x_i)=-K(\ep_i)^{-1}x_i.
\end{array}
\]

\medskip
For convenience of quotation, we also give some definitions about Hopf algebras.
\begin{definition}
Given a Hopf algebra $H$ over $k$, $\Ga\in H$ is called a {\it left integral} of $H$, if
$h\,\Ga=\ve(h)\,\Ga$, for any $h\in H$.
Similarly, we can define right integrals of $H$.
\end{definition}
Denote by $G(H)$ the group consisting of group-like elements of Hopf algebra $H$.
\begin{definition}
When $H$ is finite dimensional, there exists a unique group-like element $\al\in H^*$, satisfying
$\Ga\, h=\al(h)\,\Ga$, for any $h\in H$,
where $\Ga$ is any non-zero left integral of $H$.
$\al$ is called the {\itshape distinguished group-like element} of $H^*$.
\end{definition}

As the dual space of the finite dimensional Hopf algebra $H$, $H^*$ has a natural Hopf algebra structure, we also have the concept of left or right integral of $H^*$.
But distinct from that of $H$, it usually defines the distinguished group-like element $g\in H$ associated with a non-zero right integral $\la$ of $H^*$, that is,
$p\la=p(g)\la, \ \forall\; p\in H^*$.

Recall the notion of quasitriangular Hopf algebras (see \cite{Ka}, \cite{Mon}).
\begin{definition}
An {\itshape almost cocommutative Hopf algebra} over $k$ is a pair $(H,R)$, where $H$ is a Hopf algebra, and $R\in H\ot H$
is an invertible element, such that for arbitrary $h\in H$,
\begin{equation}\label{qc}
\tau(\De(h))=R\De(h)R^{-1},
\end{equation}
where $\tau$ is the flip map on $H\ot H$.

\smallskip
In addition, if $R$ satisfies:
$
(\De\ot\mi)R=R^{13}R^{23}$, $(\mi\ot\De)R=R^{13}R^{12}$,
$(H,R)$ is called a {\itshape quasitriangular Hopf algebra} over $k$, and $R$ is a {\itshape universal $R$-matrix} of $H$, where
$R^{12}=\sum\nolimits_ia_i\ot b_i\ot1,R^{23}=\sum\nolimits_i1\ot a_i\ot b_i,R^{13}=\sum\nolimits_ia_i\ot1\ot b_i$, if set $R=\sum\nolimits_ia_i\ot b_i$.
\end{definition}
If $(H,R)$ is quasitriangular, $R$ satisfies the quantum Yang-Baxter equation
\[R^{12}R^{13}R^{23}=R^{23}R^{13}R^{12}.\]

\begin{definition}
Given a quasitriangular Hopf algebra $(H,R)$, $v\in H$ is said a {\itshape quasi-ribbon element},
if it satisfies the conditions
\[
\begin{array}{l}
\quad v^2=us(u),
\quad s(v)=v,
\quad \ve(v)=1,
\quad \De(v)=(R^{21}R)^{-1}(v\ot v).
\end{array}
\]
where $R^{21}=\tau(R)$, and $u=\sum_is(b_i)a_i$, called the {\itshape Drinfel'd element}.
When $v$ lies in the center of $H$, it is called a {\itshape ribbon element}, and $(H,R,v)$ a {\itshape ribbon Hopf algebra} over $k$.
\end{definition}

Recall the definition of Drinfel'd double $D(H)$ for any finite dimensional Hopf algebra $H$ (see \cite[$\S 10.3$]{Mon}).
Denote
$D(H)=(H^*)^{\bs{\upshape cop}}\bowtie H$.
Its multiplication is defined as
\[(p\ot h)(q\ot k)=p\lb h_{(1)}\rightharpoonup q\leftharpoonup s^{-1}(h_{(3)})\rb\ot h_{(2)}k, \quad \forall\; h, k\in H, \; p, q\in H^*,\]
while the comultiplication coincides with that of $(H^*)^{\bs{\upshape cop}}\ot H$. Here $\rightharpoonup$ (resp. $\leftharpoonup$) denotes the left (resp. right) action of $H$ on $H^*$ given by
\[h\rightharpoonup p=p_{(2)}(h)p_{(1)}\, \quad (\text{resp. }p_{(1)}(h)p_{(2)}=p \leftharpoonup h).\]
$D(H)$ is quasitriangular, whose universal $R$-matrix is given by
\begin{equation}\label{ur}
\mathcal {R}=(\ve\ot h_i)\ot(h^i\ot 1),
\end{equation}
where $\{h_1,\ldots,h_r\}$ is a $k$-basis of $H$,
and $\{h^1,\ldots,h^r\}$ is the dual basis in $H^*$ with $r=\ms{dim}H$. Note that
$\mathcal {R}^{-1}=(\ve\ot s(h_i))\ot(h^i\ot 1)$,
where we used the Sweedler's notation:
$\De(h)=h_{(1)}\ot h_{(2)}, \ \forall\, h\in H$, as well as
 the Einstein's sum convention to omit the symbol ``$\sum$".


\begin{definition}
Given a Hopf algebra $H$, a vector space $M$ is called a {\itshape Yetter-Drinfel'd $H$-module}, if $M$ has a left $H$-module structure $(M,\cdot)$, and a right $H$-comodule structure $(M,\rho)$, such that two structures satisfy the compatible condition
\[h_{(1)}\cdot m_{(0)}\ot h_{(2)}m_{(1)}=(h_{(2)}\cdot m)_{(0)}\ot(h_{(2)}\cdot m)_{(1)}h_{(1)},
\quad \forall\; h\in H, \; m\in M\]
where $\rho(m)=m_{(0)}\ot m_{(1)}\in M\ot H$.
When $H^{\bs{\upshape op}}$ has antipode $\vs$, the condition above is equivalent to
\[\rho(h\cdot m)=h_{(2)}\cdot m_{(0)}\ot h_{(3)}m_{(1)}\vs(h_{(1)}),\quad \forall\; h\in H,\; m\in M.\]
\end{definition}

Denote by $_H\mathcal {Y}\mathcal {D}^H$ the category of the Yetter-Drinfel'd $H$-modules,
whose morphisms are homomorphisms of both left $H$-modules and right $H$-comodules. By a theorem of Majid, for a finite-dimensional Hopf algebra $H$, a $k$-vector space $M$ is a left $D(H)$-module if and only if  $M\in{_H\mathcal {Y}\mathcal {D}^H}$. As is known, $M$ has a right $H$-comodule structure $(M,\rho)$ if and only if it has a left $H^*$-module structure $(M,\cdot)$, which is defined by
\[p\cdot m=p(m_{(1)})m_{(0)},\quad \forall\; p\in H^*,\; m\in M.\]
Hence, the left $D(H)$-module structure of $M\in{_H\mathcal {Y}\mathcal {D}^H}$ is defined by
\[(p\ot h)\cdot m=p\cdot (h\cdot m)=p\lb h_{(3)}m_{(1)}\vs(h_{(1)})\rb h_{(2)}\cdot m_{(0)},\quad\forall\; h\in H, \;p\in H^*, \;m\in M.\]
In addition, for any finite-dimensional representation $(M,\pi)$ of $D(H)$, we get a solution of the Yang-Baxter equation
\[R_M=\tau\circ(\pi\ot\pi)(\mathcal {R})\in\ms{End}(M\ot M),\]
where $\tau$ is the flip map on $M\ot M$. That means
\[
\lb R_M\ot\mi_M\rb \lb\mi_M\ot R_M\rb \lb R_M\ot\mi_M\rb
=\lb\mi_M\ot R_M\rb \lb R_M\ot\mi_M\rb \lb\mi_M\ot R_M\rb.
\]
Explicitly, for $\forall\; m, n\in M$, we have
\begin{equation}
\label{eq0}
\begin{split}
R_M(m\ot n)&=(h^i\ot 1)\cdot n\ot(\ve\ot h_i)\cdot m\\
&=h^i(n_{(1)})n_{(0)}\ot h_i\cdot m=n_{(0)}\ot n_{(1)}\cdot m
\end{split}
\end{equation}

\section{Drinfel'd double of the $n$-rank Taft algebra}

When $q$ is a primitive $\ell$-th root of unity, i.e., $\ell=\ms{char}(q)<\infty$, there exists a Hopf ideal of $\mathscr{A}_q(n)$ generated by
$x_1^\ell,\cdots,x_n^\ell$, denoted by $I$ (cf. \cite{Hu2}). The quotient Hopf algebra \[\bar{\mathscr{A}}_q(n)=\mathscr{A}_q(n)/I,\]
is called the {\itshape $n$-rank Taft algebra}. When $n=1$, it is the Taft algebra. Furthermore, if $\ell=2$, it is the Sweedler algebra of dimension $4$.


Recall the Gauss integers as follows,
\[
\begin{array}{l}
(0)_q=0, \quad (0)_q!=1,\\
(m)_q=1+\cdots+q^{m-1}=(1-q^m)/(1-q),\\
(m)_q!=(1)_q\cdots(m)_q, \quad \forall\; m\in\mathbb{Z}^{>0},
\end{array}
\]
Notice that $(m)_q!\neq 0, \ 0\leq m<\ell$, and $(m)_q!= 0, \ m\geq \ell$.
One can define
\[{m\choose j}_q=\begin{cases}
\dfrac{(m)_q!}{(m-j)_q!(j)_q!},& \mbox{if} \quad 0\leq m<\ell,\\
0,&\mbox{otherwise}.\end{cases}\]

Endow $\La$ with a natural partial order $\leq$ as follows
\[\forall\; \al,\; \be\in\La, \quad \al\leq\be \quad \mbox{\it if and only if} \quad \al_i\leq\be_i, \quad  \forall\; 1\leq i\leq n.\]
For $\forall\; \al\in\La_+$,
define $\La_+^{\leq\al}=\big\{\,\be\in\La_+\mid \be\leq \al\,\big\}$,
and set
\[\ka=(\ell-1,\ldots,\ell-1), \quad\iota=(\ell,\ldots,\ell)\in\La_+.\]
Then $\big\{\,x^\ga K(\al)\mid \al, \ga\in\La_+^{\leq
\ka}\,\big\}$ is a $k$-basis of $\bar{\mathscr{A}}_q(n)$. Hence,
$\ms{dim}\bar{\mathscr{A}}_q(n)=\ell^{2n}$.

\medskip
Recall the Hopf algebra structure of $\bar{\mathscr{A}}_q(n)$.
Thanks to
\begin{equation*}
\De(x_i^{\ga_i})=\sum\limits_{\xi_i=0}^{\ga_i}{{\ga_i}\choose{\xi_i}}_q
x_i^{\ga_i-\xi_i}K(\ep_i)^{\xi_i}\ot x_i^{\xi_i},~i=1,\ldots,n,
\end{equation*}
combining with relations $\ms{(r1)}$~--~$\ms{(r4)}$, we have the explicit comultiplication formula of $\bar{\mathscr{A}}_q(n)$
\begin{equation} \label{eq:1}
\De(x^\ga K(\al))=\sum\limits_{\xi\in\La_+^{\leq\ga}}{\ga\choose\xi}
q^{-(\ga-\xi)*\xi}x^{\ga-\xi}K(\xi+\al)\ot x^\xi K(\al),
\end{equation}
where $\al,\ga\in\La_+^{\leq\ka},~\ds{{\ga\choose\al}=\prod\limits_{i=1}^n{\ga_i\choose\al_i}_q}$, together with the counit
\begin{equation} \label{eq:2}
\ve(x^\ga K(\al))=\de_{\ga,0}.
\end{equation}

Set $\pmb{1}=(1,\ldots,1)\in\La$.
By $\ms{(r1)}$~--~$\ms{(r4)}$, for $\forall\; \ga\in\La_+^{\leq\ka}$,
we can get
\begin{equation} \label{eq:3}
\begin{split}
s(x^\ga) & = (-1)^{|\ga|}q^{-\tfrac{1}{2}\lan\ga,\ga+\pmb{1}\ran}x^\ga
K(-\ga) \\
&=(-1)^{|\ga|}q^{\tfrac{1}{2}\lan\ga,\ga-\pmb{1}\ran}\te(\ga,\ga)^{-1}
K(-\ga)x^\ga,\\
s^2(x^\ga)&=q^{-|\ga|}x^\ga,
\end{split}
\end{equation}
where $|\ga|=\sum_{i=1}^n\ga_i$.

For the group-like elements of $\bar{\mathscr{A}}_q(n))$, it is easy to see that
\begin{proposition}
\label{m1}
$G(\bar{\mathscr{A}}_q(n))=\left\{K(\al)~\big|~\al\in\La_+^{\leq\ka}\right\}.$
In particular,
\begin{equation}\label{e1}
G(\bar{\mathscr{A}}_q(n))\cong G(T_\ell)^{\times n}\cong
\overbrace{\mathbb{Z}_\ell\times\cdots\times\mathbb{Z}_\ell}^n,
\end{equation}
where $T_\ell$ is the Taft algebra with dimension $\ell^2$.
\end{proposition}

Now consider the dual Hopf algebra $\bar{\mathscr{A}}_q(n)^*$,
with the canonical Hopf algebra structure induced by $\bar{\mathscr{A}}_q(n)$.
For any $\al\in\La$, $\be$, $\ga\in\La_+^{\leq\ka}$, set
\begin{gather*}
\mk(\al)\in\bar{\mathscr{A}}_q(n)^*, \qquad \mk
(\al)(x^\ga
K(\be))=\de_{\ga,0}q^{\lan\al,\be\ran},\\
\mx_i\in\bar{\mathscr{A}}_q(n)^*,\qquad \mx_i(x^\ga K
(\be))=\de_{\ga,\ep_i} \quad (i=1,\cdots,n).
\end{gather*}

It is easy to check that $\mk(\al)\in G\lb\bar{\mathscr{A}}_q(n)^*\rb$, and
\[\mk(\al)\mk(\be)=\mk(\al+\be), \quad \forall\; \al,\; \be\in\La.\eqno(*)\]
By the definition of $\mk(\al)$ and Eq. (\ref{eq:2}), we know that $\mk(0)=\ve$, which is the identity of $\bar{\mathscr{A}}_q(n)^*$. In addition, we claim that the map $\al\in\La_+^{\leq\ka}\mapsto \mk(\al)$ is injective. In fact, note that
\[\mk(\al)(K(\ep_i))=q^{\al_i},~i=1,\ldots,n.\]
Hence, $\mk(\al)=\ve,~-\ka\leq\al\leq\ka$ if and only if $\al=0$. Together with $(*)$,  the claim is true.

By Proposition \ref{m1}, we have
\begin{proposition}
\label{m2}
$G(\bar{\mathscr{A}}_q(n)^*)=\big\{\,\mk(\al)\mid\al\in\La_+^{\leq\ka}\,\big\}
\cong\overbrace{\mathbb{Z}_\ell\times\cdots\times\mathbb{Z}_\ell}^n.$
\end{proposition}
\begin{proof}
First we see that any group-like element of $\bar{\mathscr{A}}_q(n)^*$ restricted to $G(\bar{\mathscr{A}}_q(n))$ is a group character. Since for $\forall\; \chi\in G(\bar{\mathscr{A}}_q(n)^*)$,
\[\chi(K(\al)K(\be))=\De^*(\chi)(K(\al)\ot K(\be))=\chi(K(\al))\chi(K(\be)),~\chi(1)=\ve^*(\chi)=1.\]
By Dedekind's lemma \cite[p.19]{Mo} about linear independence of distinct group characters,
this means the linear independence of distinct group-like elements.
By isomorphism (\ref{e1}), the character group of $G(\bar{\mathscr{A}}_q(n))$ is exactly
$G(\bar{\mathscr{A}}_q(n)^*)=\big\{\,\mk(\al)\mid\al\in\La_+^{\leq\ka}\,\big\}
\cong\overbrace{\mathbb{Z}_\ell\times\cdots\times\mathbb{Z}_\ell}^n$.

Meanwhile, for $\,\forall\; \chi\in G(\bar{\mathscr{A}}_q(n)^*)$, as $x_i^\ell=0$, we have
\[\chi(x_i^\ell)=(\De^*)^{(\ell-1)}(\chi)(x_i^{\ot\ell})=\chi(x_i)^\ell=0.\]
Hence, $\chi(x_i)=0,~i=1,\ldots,n$, leading to
\[\chi(x^\ga)=\de_{\ga,0},~\forall \ga\in\La_+^{\leq\ka}.\]
That means $\chi$ has the form $\mk(\al)$, $\al\in\La_+^{\leq\ka}$.
\end{proof}

\begin{remark}
Both group algebras $k[G(\bar{\mathscr{A}}_q(n))]$ and $k[G(\bar{\mathscr{A}}_q(n)^*)]$
are isomorphic to $k^{\times\ell^n}$ as Hopf algebras, we can explicitly
describe the dual basis of $G(\bar{\mathscr{A}}_q(n)^*)=\big\{\,\mk(\al)\mid\al\in\La_+^{\leq\ka}\,\big\}$
in $k[G(\bar{\mathscr{A}}_q(n)^*)]\cong \lb k[G(\bar{\mathscr{A}}_q(n))]\rb^*$.
\end{remark}

Let us list explicit information on the dual Hopf algebra $(\bar{\mathscr{A}}_q(n)^*,\De^*,\ve^*,s^*)$.

The multiplication of $\bar{\mathscr{A}}_q(n)^*$ is induced from the comultiplication of $\bar{\mathscr{A}}_q(n)$ as follows
\begin{gather*}
\mk(\al)\mk(\be)=\mk(\al+\be)=\mk(\be)\mk(\al),\qquad\mk(\al)^\ell=\ve,\\
\mk(\al)\mx_i=q^{\lan\al,\ep_i\ran}\mx_i\mk(\al),\\
\mx_i\mx_j=\te(\ep_j,\ep_i)\mx_j\mx_i,\qquad \mx_i^\ell=0,
\end{gather*}
where $\al, \;\be\in\La, \ i, j=1,\ldots,n$.

Furthermore, put $\mx^\al=\mx_1^{\al_1}\cdots\mx_n^{\al_n}, \ \al\in\La_+^{\leq\ka}$,
then
\begin{equation}
\label{eq:4}
\mx^\ga(x^\eta K(\be))=\de_{\ga,\eta},\quad \forall\; \eta, \;\be\in\La_+^{\leq\ka},
\end{equation}
and
\begin{equation}
\label{eq:5}
\mx^\ga\mx^\eta=q^{-\ga*\eta}\mx^{\ga+\eta}=\te(\eta,\ga)\mx^\eta\mx^\ga,\quad \mk(\al)\mx^\ga=q^{\lan\al,\ga\ran}\mx^\ga\mk(\al).
\end{equation}
Hence, for $\,\forall\; \al,\, \be,\, \ga,\, \eta\in\La_+^{\leq\ka}$,
\[\mx^\ga\mk(\al)\lb x^\eta K(\be)\rb=q^{\lan\al,\be\ran}\de_{\ga,\eta}.\]

By Proposition \ref{m2} and Eq. (\ref{eq:4}),
$\big\{\mx^\ga\mk(\al) \mid\al,\ga\in\La_+^{\leq\ka}\,\big\}$
is a $k$-basis of $\bar{\mathscr{A}}_q(n)^*$. Hence, let $\mh(\al)\in \lb k[G(\bar{\mathscr{A}}_q(n))]\rb^*,\,\al\in\La_+^{\leq\ka}$ such that $\mh(\al)(K(\be))=\de_{\al,\be}$, we have
\begin{proposition}\label{dual}
The dual basis corresponding to
$\big\{\,x^\ga K(\al)\mid\al,\ga\in\La_+^{\leq\ka}\,\big\}$
in $\bar{\mathscr{A}}_q(n)^*$ is
\[\bigl\{\,\mx^\ga\mh(\al)\mid\al,\ga\in\La_+^{\leq\ka}\,\bigr\}.\]
\end{proposition}

Again, the comultiplication of $\bar{\mathscr{A}}_q(n)^*$ is induced from the multiplication of $\bar{\mathscr{A}}_q(n)$
\begin{gather*}
\De(\mk(\al))=\mk(\al)\ot\mk(\al),\quad \al\in\La,\\
\De(\mx_i)=\mx_i\ot\ve+\tilde\mk(\ep_i)\ot\mx_i.
\end{gather*}

\begin{remark}
If we introduce $\tilde\mk(\be)\in\bar{\mathscr{A}}_q(n)^*$, $\,\forall\, \be\in\La$, such that
\[\tilde\mk(\be)\lb x^\ga K(\al)\rb=\de_{\ga,0}\te(\al,\be)q^{\lan\al,\be\ran},\qquad \forall\; \ga,\al\in\La_+^{\leq\ka},\]
then the map $\be\in\La_+^{\leq\ka}\mapsto\tilde\mk(\be)$ may fail to be injective.
In fact,
\[\tilde\mk(\be)=\mk(\tilde\be)\in G(\bar{\mathscr{A}}_q(n)^*),\]
where we define a homomorphism of abelian groups
\[\mbox{\~{}}:\La\rightarrow\La,~\be\mapsto \tilde\be,\]
such that $\tilde\be_i=\be_1+\cdots+\be_i-\be_{i+1}-\cdots-\be_n$ $(i=1,\cdots,n-1)$, and $\tilde\be_n=\be_1+\cdots+\be_n$.
Note that $\tilde\be\equiv0~(\ms{mod}~\iota)$ if and only if $2\be_i\equiv0~(\ms{mod}~\ell)$ ($i=2,\cdots,n$) and $|\be|\equiv0~(\ms{mod}~\ell)$. Hence, when $\ell$ is odd, it holds if and only if $\be\equiv0~(\ms{mod}~\iota)$.
When $\ell$ is even, it holds if and only if there exist even numbers of $\be_i$, satisfying $\be_i\equiv\ell/2~(\ms{mod}~\ell)$, while others equal to $0$ modulo $\ell$.
The injectivity fails in the case when $\ell$ is even.
\end{remark}

Now by induction,
\begin{equation}
\begin{split}
\De(\mx^\ga\mk(\be))&=\sum\limits_{\al\in\La_+^{\leq\ga}}{\ga\choose\al}
q^{(\ga-\al)*\al}\mx^{\ga-\al}\tilde\mk(\al)\mk(\be)\ot\mx^\al\mk(\be)\\
&=\sum\limits_{\al\in\La_+^{\leq\ga}}{\ga\choose\al}
q^{(\ga-\al)*\al}\mx^{\ga-\al}\mk(\tilde\al+\be)\ot\mx^\al\mk(\be),
\end{split}
\end{equation}
while
\[\tilde\mk(\al)\mx^\ga=\te(\ga,\al)q^{\lan\ga,\al\ran}\mx^\ga\tilde\mk(\al).\]
In addition,
\[\ve^*(\mk(\al))=\mk(\al)(1)=1,\quad \ve^*(\mx_i)=\mx_i(1)=0.\]
Finally, we can give the formula of the antipode $s^*$
\[s^*(\mk(\al))=\mk(\al)\circ s=\mk(\al)^{-1}, \quad \al\in\La,\]
\[
\begin{split}
s^*(\mx^\ga)&=\mx^\ga\circ s=(-1)^{|\ga|}q^{-\tfrac{1}{2}\lan\ga,\ga+\pmb{1}\ran}\mx^\ga
\tilde\mk(-\ga)\\
&=(-1)^{|\ga|}q^{\tfrac{1}{2}\lan\ga,\ga-\pmb{1}\ran}
\tilde\mk(-\ga)\mx^\ga, \quad\ga\in\La_+^{\leq\ka}.
\end{split}
\]
We denote $s^*$ as $S$ below.

\begin{remark}
From Eq. (\ref{eq:5}), we can see that $\bar{\mathscr{A}}_q(n)^*$ contains $k[A_{\bar{q}}^{n|0}]$ as its subalgebra where $\bar{q}=q^{-1}$, and it is not isomorphic to $\bar{\mathscr{A}}_q(n)$ when $n>1$.
 That is quite different from the self-duality of the Taft algebra $T_\ell$, i.e. $T_\ell\cong {T_\ell}^*$. This owes to the skew primitive elements $x_i$'s, satisfying the relation {\sffamily (r4)}, which enriches the Hopf algebra structure of $\bar{\mathscr{A}}_q(n)$.
\end{remark}

After the preparation above, we begin to study the Drinfel'd double of $\bar{\mathscr{A}}_q(n)$,
\[D\lb\bar{\mathscr{A}}_q(n)\rb=\lb\bar{\mathscr{A}}_q(n)^*\rb^\bs{cop}\bowtie \bar{\mathscr{A}}_q(n),\]
briefly denoted by $D(\bar{\mathscr{A}})$. To end this section, we consider the question when $D(\bar{\mathscr{A}})$ is a ribbon Hopf algebra.

By \cite[Thm. 3]{KR1}, if $H$ is a finite dimensional Hopf algebra over $k$, with the antipode $s$,
and $g,~\al$ is the distinguished group-like element of $H, \,H^*$, respectively, then

(a) $(D(H),\mathcal {R})$ has a quasi-ribbon element if and only if there exist $l\in G(H)$ and $\be\in G(H^*)$, such that $l^2=g$ and $\be^2=\al$.

(b) $(D(H),\mathcal {R})$ has a ribbon element if and only if there exist $l\in G(H)$ and $\be\in G(H^*)$, satisfying conditions in (a), such that
\[s^2(h)=l(\be\rightharpoonup h\leftharpoonup\be^{-1})l^{-1},\quad \forall\, h\in H.\]

\smallskip
In addition, by \cite[Cor. 3]{KR1},
if $H$ is a finite dimensional Hopf algebra over $k$,
and $G(H)$, $G(H^*)$ are both of odd order,
 $(D(H),\mathcal{R})$ has a ribbon element (necessarily unique) if and only if $s^2$ is an automorphism of odd order.

\smallskip
Especially, for the Taft algebra $T_\ell$, $G(T_\ell)$ is a cyclic group of $\ell$ order generated by the distinguished  group-like element, so is $G(T_\ell^*)$.
 Besides, $s^2$ is also of $\ell$ order.
 Hence, when $\ell$ is odd, $D(T_\ell)$ becomes a ribbon Hopf algebra.
 Finally, by \cite[Prop. 7]{KR1}, $D(T_\ell)$ is a ribbon Hopf algebra if and only if $\ell$ is odd.

\smallskip
Now we turn to the $n$-rank Taft algebras. Refer the proof of \cite[Prop. 7]{KR1} to give a more general conclusion.
For simplicity, we omit the summation range indexed on the $\Sigma$'s below, with the default option, summation over $\La_+^{\leq\ka}$.

First, we need to compute the left integral of $\bar{\mathscr{A}}_q(n)$, the right integral of $\bar{\mathscr{A}}_q(n)^*$, and also the distinguished group-like elements.
Take any $\Ga=\sum\nolimits_{\ga,\al}a_{\ga,\al}x^{\ga}K(\al)\in \bar{\mathscr{A}}_q(n)$, and let
\begin{gather*}
K(\be)\,\Ga=\sum\nolimits_{\ga,\al}a_{\ga,\al}\te(\be,\ga)q^{\lan\be,\ga\ran}x^\ga K(\al+\be)=\ve(K(\be))\Ga=\Ga,\quad \forall\, \be\in\La_+^{\leq\ka},\\
x^\eta\,\Ga=\sum\nolimits_{\ga,\al}a_{\ga,\al}q^{\eta*\ga}x^{\ga+\eta}K(\al)=\ve(x^{\eta})\Ga=\de_{\eta,0}\Ga,\quad \forall\, \eta\in\La_+^{\leq\ka}.
\end{gather*}
By the second equation above, we see that $a_{\ga,\al}=0, \quad \forall\, \ga\neq\ka$. It follows from the first equation that
$a_{\ka,\al}=a_{\ka,0}\te(\al,\ka)q^{\lan\al,\ka\ran}$, $\forall\, \al\in\La_+^{\leq\ka}$.
Hence, $\bar{\mathscr{A}}_q(n)$ has a non-zero left integral
\begin{equation}
\Ga=x^\ka\sum\nolimits_\al\te(\al,\ka)q^{\lan\al,\ka\ran}K(\al).
\end{equation}

Similarly, we obtain a non-zero right integral of $\bar{\mathscr{A}}_q(n)^*$
\begin{equation}
\la=\mx^\ka\sum\nolimits_\al\mk(\al),
\end{equation}
and the distinguished group-like elements of $\bar{\mathscr{A}}_q(n)$ and  $\bar{\mathscr{A}}_q(n)^*$ are respectively given by
\begin{equation*}
\begin{split}
\Ga  x^\eta K(\be)&=x^\ka\sum\nolimits_\al\te(\al,\ka+\eta)q^{\lan\al,\ka+\eta\ran}x^\eta K(\al+\be)\\
&=x^{\ka+\eta}\sum\nolimits_\al\te(\al,\ka+\eta)q^{\lan\al,\ka+\eta\ran+\ka*\eta}K(\al+\be)\\
&=\de_{\eta,0}\te(\ka,\be)q^{-\lan\be,\ka\ran}\Ga=\de_{\eta,0}\te(\be,\pmb{1})q^{\lan\be,\bs{\bfseries
1}\ran}\Ga\\
&=\tilde\mk(\pmb{1})\lb x^\eta K(\be)\rb\Ga,
\end{split}
\end{equation*}
\begin{equation*}
\begin{split}
 \mx^\eta\mk(\be)\la &=q^{\lan\ka,\be\ran}\mx^\eta\mx^\ka\sum\nolimits_\al\mk(\al+\be)\\
&=q^{\lan\ka,\be\ran-\eta*\ka}\mx^{\eta+\ka}\sum\nolimits_\al\mk(\al+\be)\\
&=\de_{\eta,0}q^{\lan\ka,\be\ran}\la\\
&=\mx^\eta\mk(\be)\lb K(\ka)\rb\la.
\end{split}
\end{equation*}
Hence, the distinguished group-like element of $\bar{\mathscr{A}}_q(n)$ is $K(\ka)$, while the one of $\bar{\mathscr{A}}_q(n)^*$ is $\tilde\mk(\pmb{1})$.
Now we have
\begin{theorem}\label{ribb}
The Drinfel'd double $D(\bar{\mathscr{A}})$ of $\bar{\mathscr{A}}_q(n)$ is a ribbon Hopf algebra if and only if $\ell$ is odd.
In this situation, $D(\bar{\mathscr{A}})$ has the unique ribbon element
\begin{equation}\label{rib}
v=u\lb\tilde\mk\lb(\iota+\pmb{1})/2\rb^{-1}\ot K(\ka/2)^{-1}\rb=u\lb\tilde\mk(\ka/2)\ot K\lb(\iota+\pmb{1})/2\rb\rb,
\end{equation}
where $u\in D(\bar{\mathscr{A}})$ is the Drinfel'd element.
 \end{theorem}
\begin{proof}
$\Rightarrow)$ If $D(\bar{\mathscr{A}})$ is a ribbon Hopf algebra, it has a quasi-ribbon element. By \cite[Thm. 3 (a)]{KR1}, it is equivalent to that there exist $K(\al)\in G(\bar{\mathscr{A}}_q(n))$ and $\mk(\be)\in G(\bar{\mathscr{A}}_q(n)^*)$, such that $K(\al)^2=K(2\al)=K(\ka)$ and $\mk(\be)^2=\mk(2\be)=\tilde\mk(\pmb{1})=\mk(\tilde{\pmb{1}})$.
Therefore,
\[
2\al\equiv\ka~(\ms{mod}~\iota),\qquad
2\be\equiv\tilde{\pmb{1}}~(\ms{mod}~\iota).
\]
From the first congruence above, we know that $\ell$ is odd.

$\Leftarrow)$ For any finite dimensional Hopf algebra $H$, by the Nichols-Zoeller Theorem \cite[$\S$ 3]{Mon},
together with the linear independence of distinct elements in $G(H)$, $\ms{dim}~H$ is divisible by $|G(H)|$.
  Now $\ms{dim}~\bar{\mathscr{A}}_q(n)=\ell^{2n}$, thus $\ell^{2n}$ is divisible by $|G(\bar{\mathscr{A}}_q(n))|=|G(\bar{\mathscr{A}}_q(n)^*)|$ (In fact, by Props. \ref{m1}, \ref{m2}, $|G(\bar{\mathscr{A}}_q(n))|=|G(\bar{\mathscr{A}}_q(n)^*)|=\ell^n$).
By Eq. (\ref{eq:3}), $s^2$ is also of order $\ell$. Hence, if $\ell$ is odd, $G(\bar{\mathscr{A}}_q(n))$,
$G(\bar{\mathscr{A}}_q(n)^*)$ and $s^2$ are all of odd order. By \cite[Cor. 3]{KR1}, $D(\bar{\mathscr{A}})$ is a ribbon Hopf algebra.

Next, we compute the unique ribbon element of $D(\bar{\mathscr{A}})$ in the case when $\ell$ is odd. By \cite[Thm. 3]{KR1}, we only need to solve $K(\al_0)\in G(\bar{\mathscr{A}}_q(n)),~\mk(\be_0)\in G(\bar{\mathscr{A}}_q(n)^*)$, satisfying
\begin{gather*}
K(\al_0)^2=K(\ka),\qquad \mk(\be_0)^2=\tilde\mk(\pmb{1}),\\
s^2\lb x^\ga K(\al)\rb=K(\al_0)\lb\mk(\be_0)\rightharpoonup x^\ga K(\al)\leftharpoonup\mk(\be_0)^{-1}\rb K(\al_0)^{-1},\quad \forall\, \al, \ga\in\La_+^{\leq\ka}.
\end{gather*}
As $\ell$ is odd, by the first two equations above we have
\[K(\al_0)=K(\ka/2),\qquad\mk(\be_0)=\tilde\mk\lb(\iota+\pmb{1})/2\rb.\]
Finally, we need to check the third equation,
\[
\begin{split}
K(\ka/2)&\lb\tilde\mk\lb(\iota+\pmb{1})/2\rb\rightharpoonup x^\ga K(\al)\leftharpoonup\tilde\mk\lb(\iota+\pmb{1})/2\rb^{-1}\rb K(\ka/2)^{-1}\\
&=\tilde\mk\lb-(\iota+\pmb{1})/2\rb\lb K(\ga)\rb K(\ka/2)x^\ga K(\al)K(-\ka/2)\\
&=\te(\ga,\ka/2)q^{\lan\ga,\ka/2\ran}\te(\ka/2,\ga)q^{\lan\ga,\ka/2\ran}x^\ga K(\al)\\
&=q^{2\lan\ga,\ka/2\ran}x^\ga K(\al)=q^{-|\ga|}x^\ga K(\al)\\
&=s^2\lb x^\ga K(\al)\rb.
\end{split}
\]
Then by \cite[Thm. 1]{KR1}, we obtain the desired ribbon element $v$ of $D(\bar{\mathscr{A}})$.
\end{proof}

In the end of this section, we show that the $n$-rank Taft algebra $\bar{\mathscr{A}}_q(n)$ is 2-cocycle twist equivalent to the $n$-fold tensor product of Taft algebra, namely $\bar{\mathscr{A}}_q(1)^{\ot n}$ (the {\it rank $n$} Taft algebra), generalizing the special case when $n=2,\,q=-1$ in \cite{LH}, where one can see that Hopf algebras 2-cocycle twist equivalent to each other may be non-isomorphic, and with quite different representation theory.

First of all,
let us briefly recall the definition of 2-cocycle twist of a Hopf algebra.
Associated with a 2-cocycle $\si$ as a bilinear form defined on a bialgebra $H$,
which is invertible under the convolution product and satisfies
\begin{gather*}
\si(a,1)=\si(1,a)=\ve(a),\quad a\in H,\\
\si(a_{(1)},b_{(1)})\,\si(a_{(2)}b_{(2)},c)=\si(b_{(1)},c_{(1)})\,\si(a,b_{(2)}c_{(2)}),\quad a,\, b,\, c\in H,
\end{gather*}
one can construct a new bialgebra $(H^\si,\cdot_\si,\De,\ve)$ with
\[a\cdot_\si b=\si(a_{(1)},b_{(1)})\,a_{(2)}b_{(2)}\,\si^{-1}(a_{(3)},b_{(3)}),\quad a,\, b\in H.\]
Moreover, if $H$ is a Hopf algebra with the antipode $S$, then so is $H^\si$ with its antipode $S^\si$ given by
$S^\si(a)=\lan \mathscr{U},a_{(1)}\ran\, S(a_{(2)})\,\lan \mathscr{U}^{-1},a_{(3)}\ran$, for $a\in H^\si$,
where $\mathscr{U}=\si(\mi\ot S)\De\in H^*$ with the inverse $\mathscr{U}^{-1}=\si^{-1}(S\ot\mi)\De$.

\begin{theorem}\label{twist}
There exists a natural Hopf algebra isomorphism
\[\bar{\mathscr{A}}_q(n)\cong\lb\bar{\mathscr{A}}_q(1)^{\ot n}\rb^\si\]
with the corresponding Hopf 2-cocycle $\si$ on $\bar{\mathscr{A}}_q(1)^{\ot n}$ defined by:
\[\si(u,v)=
\begin{cases}
q^{-\be*\al},&u=K(\al)\text{ and }v=K(\be),\,\al,\be\in\La,\\
0,& u\text{ or }v\notin k\lan K(\al)\mid\al\in\La\ran.
\end{cases}\]
\begin{proof}
Here we  write the basis vectors in $\bar{\mathscr{A}}_q(1)^{\ot n}$ as
\[x^\ga K(\al):=x_1^{\ga_1}K(\al_1\ep_1)\ot\cdots\ot x_1^{\ga_n}K(\al_n\ep_1),\,\al,\ga\in\La_+^{\leq
\ka},\]
and use $\cdot$ to denote the multiplication of $\bar{\mathscr{A}}_q(1)^{\ot n}$. Then the desired Hopf algebra isomorphism identifies $k$-bases $\{x^\ga K(\al)\}_{\al,\ga\in\La_+^{\leq
\ka}}$ in both $\bar{\mathscr{A}}_q(n)$ and $\bar{\mathscr{A}}_q(1)^{\ot n}$.

Indeed, $\si$ is clearly a Hopf 2-cocycle. One only need to check that the twisted algebra $\lb\lb\bar{\mathscr{A}}_q(1)^{\ot n}\rb^\si,\cdot_\si\rb$ satisfies all the defining relations of $\bar{\mathscr{A}}_q(n)$ straightforwardly. For example, we have
\begin{align*}
K(\ep_i)\cdot_\si x_j&= \si(K(\ep_i),K(\ep_j))\, K(\ep_i)\cdot x_j\, \si^{-1}(K(\ep_i),1)
=q^{-\ep_j*\ep_i}\,K(\ep_i)\cdot x_j\\
&=\te(\ep_i,\ep_j)q^{-\ep_i*\ep_j}q^{\de_{ij}}\,x_j\cdot K(\ep_i)=\te(\ep_i,\ep_j)q^{\de_{ij}}\,x_j\cdot_\si K(\ep_i),\\
x_i\cdot_\si x_j &= \si(K(\ep_i),K(\ep_j))\,x_i\cdot x_j\,\si^{-1}(1,1)=q^{-\ep_j*\ep_i}\,x_i\cdot x_j\\
&= \te(\ep_i,\ep_j)q^{-\ep_i*\ep_j}\,x_j\cdot x_i=\te(\ep_i,\ep_j)\,x_j\cdot_\si x_i,
\end{align*}
for any $i,j=1,\dots,n$.
\end{proof}

\end{theorem}

\section{The knot invariants associated with the twist oriented quantum algebras}

In this section,  we first recall the oriented quantum algebras (\cite{KR, Rad, Rad2}) and the resulting knot invariants, then apply to our case $D(\bar{\mathscr{A}})$.

\noindent
\begin{definition}
\label{oqa}
An {\itshape oriented quantum algebra} over $k$ is a tuple
$(A, \rho, D, U)$,
where $A$ is an algebra over $k$, $\rho\in A\ot A$ is an invertible element,
and $D,~U$ are two commuting algebra automorphisms of $A$, satisfying
\[\begin{array}{ll}
\ms{\upshape (qa.1)}~&(\ms{\upshape Id}_A\ot
U)(\rho)~\mbox{and}~(D\ot\ms{\upshape Id}_A
)(\rho^{-1})~\mbox{are invertible in}~A\ot A^\bs{op},\\
\ms{\upshape (qa.2)}~&\rho=(D\ot D)(\rho)=(U\ot U)(\rho), \\
\ms{\upshape (qa.3)}~&\rho_{12}\rho_{13}\rho_{23}=\rho_{23}\rho_{13}\rho_{12},
\end{array}
\]
where $\rho=\sum\limits_ia_i\ot b_i$, and
$\rho_{12}=\sum\limits_ia_i\ot b_i\ot 1$,
$\rho_{13}=\sum\limits_ia_i\ot 1\ot
b_i$, $\rho_{23}=\sum\limits_i1\ot a_i\ot b_i$.

A {\itshape twist oriented quantum algebra} (in brief, TOQA) over $k$ is a tuple $(A, \rho, D, U, G)$,
where $(A, \rho, D, U)$ is an oriented quantum algebra and $G\in
A$ is invertible, such that $(D\circ U)(x)=GxG^{-1}$ for all $x\in A$.
\end{definition}

Now suppose $(H, R)$ is a quasitriangular Hopf algebra over $k$, and let $R=\sum_i a_i\ot b_i$.
The Drinfel'd element $u=\sum_i s(b_i)a_i$ is invertible, with $u^{-1}=\sum_i
b_is^2(a_i)$, and $s^2(x)=uxu^{-1}$ for all $x\in H$.

\begin{remark}
From \cite[Prop. 4.1]{RW} \& \cite{KR}, we know that if $H$ is a quasitriangular Hopf algebra over $k$,
$(H, R, \ms{\upshape Id}_H, s^{-2}, u^{-1})$
is a twist oriented quantum algebra. Moreover, if $H$ has a ribbon element $v$, instead we can consider
a twist oriented quantum algebra $(H, R, \ms{\upshape Id}_H, s^{-2}$, $u^{-1}v)$.
\end{remark}
Hence, we have a twist oriented quantum algebra
$\bigl(D(\bar{\mathscr{A}}), \mathcal
{R}, \mbox{Id}_{D(\bar{\mathscr{A}})}, S^2\ot s^{-2}, u^{-1}\bigr)$
 over $k$, or say
$\bigl( D(\bar{\mathscr{A}}), \, \mathcal
{R}, \,\mbox{Id}_{D(\bar{\mathscr{A}})},\, S^2\ot s^{-2},\, u^{-1}v\bigr)$,
where $\mathcal {R}=\sum_{i=1}^r(\ve\ot h_i)\ot(h^i\ot
1)$ and
$$
u=\sum_{i=1}^rS^{-1}(h^i)\ot h_i=\sum_{i=1}^rh^i\ot
s^{-1}(h_i),\qquad
u^{-1}=\sum_{i=1}^rS^2(h^i)\ot h_i=\sum_{i=1}^rh^i\ot s^2(h_i).
$$
Here $\{h_1,\ldots,h_r\}$ is a $k$-basis of $\bar{\mathscr{A}}_q(n)$,
 while $\{h^1,\ldots,h^r\}$ is the dual basis of $\bar{\mathscr{A}}_q(n)^*$ with
$r=\ms{dim}\bar{\mathscr{A}}_q(n)=\ell^{2n}$.

In particular, by Proposition \ref{dual}, we can choose the bases $\bigl\{x^\ga K(\al)\mid\al,\ga\in\La_+^{\leq\ka}\bigr\}$ and
$\bigl\{\mx^\ga\mh(\al)\mid \al,\ga\in\La_+^{\leq\ka}\bigr\}$ dual to each other to give the formula of the universal $R$-matrix
\begin{equation}
\mathcal {R}=\sum\limits_{\al,\ga\in\La_+^{\leq\ka}}\lb\ve\ot x^\ga K(\al)\rb\ot\lb\mx^\ga\mh(\al)\ot1\rb.
\end{equation}

\subsection{Regular isotopy invariants from TOQA's}
Next, we simply introduce the construction of regular isotopy invariants arising from TOQA's in \cite{KR}.
For the basic knowledge of knot theory, one can refer to \cite{Kau1}. For the material about regular isotopy invariants, one can consult \cite{Kau}.

Given a TOQA $(A, \rho, D, U, G)$,
the first step is to construct a regular isotopy invariant of oriented 1-1 tangle diagrams,
\[\mb{Inv}_A:\mb{Tang}\rightarrow A,\]
where \nb{Tang} is the collection of all oriented 1-1 tangle diagrams.
Figure 1 is a simple example of oriented 1-1 tangle.
\[
\xy 0;/r.1pc/: (-10,10)*{}="a";(0,0)*{}="b";(0,25)*{}="c";
(-10,-10)*{}="d"; (0,-20)*{}="e";
"a";"b"**\crv{(-15,2)&(-3,3)};\POS?(.6)*{\hole}="x";
"c";"x"**\crv{(10,10)};;\POS?(.5)*{\hole}="y";;?(.25)*\dir{>};
"a";"y"**\crv{(-5,20)};"x";"d"**\crv{(-15,0)};
"d";"y"**\crv{(-5,-17)&(17,7)};\POS?(.5)*{\hole}="z";
"b";"z"**\crv{(2,-1)};"z";"e"**\crv{(6,-15)};?(.5)*\dir{>};
\endxy
\]

\centerline{\figurename\quad1}

In order to obtain \nb{Inv}, Kauffman and Radford (\cite{KR}) first construct a formal product $\mb{W}_A(\mb{T}),\mb{T}\in\mb{Tang}$,
then they turn to the specialization $\mb{w}_A(\mb{T})\in A$,
and take $\mb{Inv}_A(\mb{T})=\mb{w}_A(\mb{T})$.

The formal product $\mb{W}_A(\mb{T})$ is defined by an elaborate algorithm, which involves sliding labeled beads.
In detail, for $\mb{T}\in\mb{Tang}$, if it has no crossings,
let $\mb{W}_A(\mb{T})=1$. Otherwise, if $\mb{T}$ has $m\geq1$ crossings, traverse it in the direction of orientation,
and label the crossing lines $1,\ldots,2m$, 
successively.
Also give line $i$ the decoration $x_i$, $1\leq i\leq 2m$.
Here is an example of oriented 1-1 tangle completing decoration.
\[
\xy 0;/r.1pc/: (-10,30)*{}="a";(10,30)*{}="b";(-10,10)*{}="c";
(10,10)*{}="d"; "b";"c"**\crv{};?(1)*\dir{<};\POS?(.5)*{\hole}="x";
"d";"x"**\crv{};"a";"x"**\crv{};?(1)*\dir{<};
(-10,-10)*{}="a";(10,-10)*{}="b";
"a";"d"**\crv{};\POS?(.5)*{\hole}="x";?(1)*\dir{>};
"b";"x"**\crv{};"c";"x"**\crv{};?(1)*\dir{<}; (-10,-30)*{}="c";
(10,-30)*{}="d"; "b";"c"**\crv{};?(1)*\dir{<};\POS?(.5)*{\hole}="x";
"d";"x"**\crv{};"a";"x"**\crv{};?(1)*\dir{<};
(-10,30)*{}="a";(10,30)*{}="b";
"a";(-10,40)**\dir{-};?(1)*\dir{>};"c";(-10,-40)**\dir{-};?(0)*\dir{<};
"b";(30,30)**\crv{(10,40)&(30,40)};
"d";(30,-30)**\crv{(10,-40)&(30,-40)};
(30,-30);(30,30)**\dir{-};?(.5)*\dir{<};
(-6,14)*{\bullet}+(-5,5)*{x_3};(6,14)*{\bullet}+(5,5)*{x_6};
(-6,-6)*{\bullet}+(-5,5)*{x_5};(6,-6)*{\bullet}+(5,5)*{x_2};
(-6,-26)*{\bullet}+(-5,5)*{x_1};(6,-26)*{\bullet}+(5,5)*{x_4};
\endxy
\]

\centerline{\figurename\quad2}

Apart from the crossings, there are the following blocks left for an oriented tangle diagram, i.e., four oriented local extrema

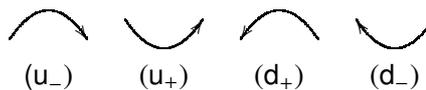
\[\xy 0;/r.12pc/:
(-55,0);(-35,0);**\crv{(-45,15)};?(0)*\dir{<};(-25,5);(-5,5);**\crv{(-15,-10)};?(0)*\dir{<};
(5,0);(25,0);**\crv{(15,15)};?(1)*\dir{>};(35,5);(55,5);**\crv{(45,-10)};?(1)*\dir{>};
(-45,-10)*{\ms{(u}_-)};(-15,-10)*{(\ms{u}_+)};(15,-10)*{(\ms{d}_+)};(45,-10)*{(\ms{d}_-)};
\endxy
\]

For $1\leq i\leq 2m$,
let $u_{\tiny\ms{d}}(i)$ be the number of local extrema of type $(\ms{d}_+)$
minus the number of type $(\ms{d}_-)$ encountered on the portion of the traversal from line $i$ to the end.
Similarly, $u_{\tiny\ms{u}}(i)$ is defined.
Now set
\[\mb{W}_A(\mb{T})=\lb D^{u_{\tiny\ms{d}}(1)}\circ U^{u_{\tiny\ms{u}}(1)}\rb(x_1)\cdots
\lb D^{u_{\tiny\ms{d}}(2m)}\circ U^{u_{\tiny\ms{u}}(2m)}\rb(x_{2m}).\]

When turning to the specialization of $\mb{W}_A(\mb{T})$ in $A$, the decoration of the crossings is modified as
\begin{equation}\label{cross}
\xy 0;/r.12pc/:
(-10,-10);(-30,10)**\crv{};\POS?(.5)*{\hole}="x";?(1)*\dir{>};
(-30,-10);"x"**\crv{};"x";(-10,10)**\crv{};?(0)*\dir{>};
(10,-10);(30,10)**\crv{};\POS?(.5)*{\hole}="x";?(1)*\dir{>};
(30,-10);"x"**\crv{};"x";(10,10)**\crv{};?(0)*\dir{>};
(-25,5)*{\bullet}+(-8,-1)*{E};(-15,5)*{\bullet}+(8,-1)*{E'};
(15,-5)*{\bullet}+(-8,1)*{e};(25,-5)*{\bullet}+(8,1)*{e'};
(-20,-15)*{L_-};(20,-15)*{L_+};
\endxy
\end{equation}
Here $\rho^{-1}=E\ot E',~\rho=e\ot e'$.
As we demand $\mb{Inv}_A$ to be a regular isotopy invariant,
the decoration of other kind of oriented crossings should be computed from two standard cases above,
by adding some local extrema. For instance, we have
\[\xy 0;/r.12pc/:
(-10,-10);(-30,10)**\crv{};?(1)*\dir{<};(-20,0)*{\hole}="x";
(-30,-10);"x"**\crv{};"x";(-10,10)**\crv{};?(0)*\dir{>};
(10,-10);(30,10)**\crv{};\POS?(.5)*{\hole}="x";?(1)*\dir{>};
(30,-10);"x"**\crv{};?(1)*\dir{<};"x";(10,10)**\crv{};
(-25,5)*{\bullet}+(-9,0)*{e};(-25,-5)*{\bullet}+(-12,0)*{U(e')};
(25,5)*{\bullet}+(12,0)*{U(E)};(25,-5)*{\bullet}+(9,0)*{E'};
\endxy
\]

\medskip
Having changed $W_A(\mb{T})$ to $\mb{Inv}_A(\mb{T})$, one can be transferred to the case of oriented link diagrams.

Any oriented knot diagram can be viewed as an oriented 1-1 tangle diagram with ends joined.
 And an oriented link diagram consists of one or more oriented knot diagrams, the components of the link diagram.
For example, if we join the ends of the tangle in Fig. 1, it becomes a left-handed Trefoil knot.

By \cite[Theorem 3]{KR}, in order to convert an invariant of oriented tangles to the one of links, besides the twist element $G$, it still needs a $k$-valued trace-like function \ms{tr} on $A^*$, satisfying $\ms{tr}\circ D=\ms{tr}=\ms{tr}\circ U$. Here $f\in
A^*$ is a trace-like function if and only if $f(ab)=f(ba), \forall\, a, b\in A$.
 To each component we have to choose a base point for computation.  Using $G$ to modify the trace-like function, one can show that the corresponding invariant is independent of the choice of base points.

Now let \nb{Link} be the collection of all oriented link diagrams.
Assume that $\mb{L}\in\mb{Link}$ has components $\mb{L}_1,\ldots,\mb{L}_r$, then one can define the formal product $\mb{W}(\mb{L}_\ell)$ as follow.

When $\mb{L}_\ell$ has no crossings, let $\mb{W}(\mb{L}_\ell)=1$. Otherwise,
one can choose a base point $P_\ell$ on a vertical line of $\mb{L}_\ell$ for topological reasons.
The bead starts from $P_\ell$,  slides in the direction of orientation, and stops when back to $P_\ell$.
 The crossing lines encountered on the traversal are successively labeled as $(\ell:1),\ldots,(\ell:m_\ell)$.
For $1\leq i \leq m_\ell$, we define $u_{\tiny\ms{u}}(\ell:i)$ and $u_{\tiny\ms{d}}(\ell:i)$, similarly to the case of tangles. The only change is the involved portion of the traversal, now just from crossing line $(\ell:i)$ to base point $P_\ell$.
Let $x_{(\ell:i)}$ be the decoration on crossing line $(\ell:i)$, and set
\[\mb{W}(\mb{L}_\ell)=\lb D^{u_{\tiny\ms{d}}(\ell:1)}\circ
U^{u_{\tiny\ms{u}}(\ell:1)}\rb(x_{(\ell:1)})\cdots
 \lb D^{u_{\tiny\ms{d}}(\ell:m_\ell)}\circ
 U^{u_{\tiny\ms{u}}(\ell:m_\ell)}\rb(x_{(\ell:m_\ell)}).\]
Similarly, using $\rho$ and $\rho^{-1}$ to get the specialization $\mb{w}(\mb{L}_\ell)$ in $A$,  one then defines the invariant of oriented links as
\begin{equation}\label{eq1}
\mb{Inv}_{A,\scriptsize\ms{tr}}(\mb{L})=\ms{tr}\lb
G^{\scriptsize\ms{Wd}(\mb{L}_1)}\mb{w}(\mb{L}_1)\rb \cdots\ms{tr}\lb
G^{\scriptsize\ms{Wd}(\mb{L}_r)}\mb{w}(\mb{L}_r)\rb,
\end{equation}
where $\ms{Wd}(\mb{L}_i)$ is the {\it Whitney degree} of $\mb{L}_i$,
 i.e., the sum of the rotation numbers of its local extrema.
 The function \ms{Wd} is an elementary regular invariant of oriented link diagrams.
 The rotation number of the local extrema of type $(\ms{u}_-)$ and $(\ms{d}_-)$ are $1/2$,
while those of the other two types $(\ms{u}_+)$ and $(\ms{d}_+)$ are $-1/2$.

\medskip

In order to construct the invariant \nb{Inv} of oriented links,
one needs a trace-like function $\ms{tr}$ on $D(H)$ for any Hopf algebra $H$ over $k$.
Note that a trace-like function on $D(H)$ is just a cocommutative element in $D(H)^*$, where $H$ is finite dimensional.

In particular, since $D(H)^*$ is isomorphic to $(H^{*{\bs{cop}}}\ot
H)^*=H^*\ot H^{\bs{op}}=\mbox{Hom}(H,H^{\bs{op}})$ as an algebra, there exists an algebra isomorphism
\[\ms{f}:\mbox{Hom}(H,H^{\bs{op}})\rightarrow D(H)^*,
\quad \ms{f}(T)(p\ot a)=p(T(a)), \quad \forall\; T\in \mbox{End}(H), \ p\in
H^*, \ a\in H,\]
where $\mbox{Hom}(H,H^{\bs{op}})$ has the canonical convolution algebra structure.

By \cite[Thm. 2.2]{RW}, $\ms{f}(T)$ is a cocommutative element in $D(H)^*$, if and only if
\begin{gather*}
l(a_{(2)})\circ T\circ r(a_{(1)})=r(a_{(1)})\circ T\circ l(a_{(2)}),\quad\textit{and}\\
l(p_{(2)})\circ T^*\circ r(p_{(1)})=r(p_{(1)})\circ T^*\circ l(p_{(2)}),
\end{gather*}
hold for $\forall a\in H, \; p\in H^*$. Here $l$ (resp. $r$) is the left (resp., right) multiplication transformation of $H$ (or $H^*$).

On the other hand, as a $k$-vector space, $D(H)=H^*\ot H$ can be identified with $\mbox{End}(H)$.
Let $\ms{Tr}$ be the usual trace on $\mbox{End}(H)$, then
\[\ms{f}(T)(p\ot a)=p(T(a))=
\ms{Tr}(p\ot T(a))=\ms{Tr}(T\circ (p\ot a)).\]
Hence,
\begin{equation} \label{eq2}
\ms{f}(T)\lb T'\rb=\ms{Tr}\lb T\circ T'\rb,\qquad \forall\; T'\in \mbox{End}(H)=D(H).
\end{equation}

From the discussion above, the usual candidates for a trace-like function on $D(H)$ are
\[\ms{f}(\mbox{Id}_H)=\ms{Tr}, \qquad \ms{f}(s^{-1})=\ms{Tr}\circ s^{-1}.\]

\subsection{Ambient knot invariants from ribbon Hopf algebras}
 In general, to directly compute $\mb{Inv}_{H,\scriptsize\ms{tr}}$ even for some usual link $\mb{L}$ when $H=\bar{\mathscr{A}}_q(n)$ is rather difficult.
 In order to jump out of the dilemma, we turn to representations of $H$, since quantum invariants often originate from representations of quantum groups.
In fact, around early 1990, Reshetinkhin, Turaev, etc. had given the prototype of quantum invariants, i.e., operator invariants, see \cite{Tu}, \cite{Tu1}, \cite{RT1}. That is, given any semisimple Lie algebra $\mathfrak{g}$, and a representation $(\pi,V)$ of its quantized universal enveloping algebra $U_q(\mathfrak{g})$, one can construct the corresponding operator invariant, called the quantum $(\mathfrak{g},\pi)$ invariant in \cite{Oht}. For example, the regular invariant, Kauffman bracket, can be derived from the quantum $(\mathfrak{sl}_2,V_2)$ invariant, where  $V_2$ is the two dimensional fundamental representation of $U_q(\mathfrak{sl}_2)$.

Given a ribbon Hopf algebra $(H,R,v)$ and its finite-dimensional irreducible representation $(\pi,M)$, we modify the invariant $\mb{Inv}_{H,\scriptsize\ms{tr}}$ to be
\begin{equation}
\label{eq:6}
\mb{Inv}_{H,M}(\mb{L})=\chi_M(G^{\scriptsize\ms{Wd}(\mb{L}_1)}\mb{w}(\mb{L}_1))\cdots\chi_M(G^{\scriptsize\ms{Wd}(\mb{L}_r)}\mb{w}(\mb{L}_r)),
\end{equation}
where $G=u^{-1}v$ and $\chi_M:=\ms{Tr}\circ\pi$ is just the character of the representation $(\pi,M)$, of course a trace-like function on $H$.

From \cite{Rad}, we know that for a given representation $\pi:H\rightarrow \ms{End}(M)$, one can induce an oriented quantum algebra structure
\[\lb\ms{End}(M),\, \pi\ot\pi(R),\, \mi_M,\, \mbox{\upshape Int}_{\pi(u^{-1}v)}\rb\]
on $(\ms{End}(M), \pi\ot\pi(R))$, where $\mbox{\upshape Int}_{\pi(u^{-1}v)}$ is the inner automorphism of $\ms{End}(M)$ defined by conjugating $\pi(u^{-1}v)$. For the TOQA
\[\lb\ms{End}(M),\, \pi\ot\pi(R),\, \mi_M,\, \mbox{\upshape Int}_{\pi(u^{-1}v)},\, \pi(u^{-1}v)\rb,\]
we see that the invariant $\mb{Inv}_{H,M}$ is just $\mb{Inv}_{\scriptsize\ms{End}(M),\ms{Tr}}$.

Meanwhile, this modification has its great benefit since one can normalize the regular invariant $\mb{Inv}_{H,M}$ to an ambient one. First, in order to distinguish knot types, we need the following result (see \cite[Thm. 3.8.4]{Cro})

{\it Two diagrams of an oriented link are regular isotopic to each other if and only if they have the same writhe $\ms{\upshape Wr}$ and Whitney degree $\ms{\upshape Wd}$}.

The {\it writhe} $\ms{Wr}$ of an oriented link diagram is defined as the number of its positive crossings minus the number of its negative crossings. That is, the crossing $L_-$ (resp. $L_+$) in (\ref{cross}) contributes $-1$ (resp. $1$) to the writhe of the link.

Note that the parity of $\ms{Wr}+\ms{Wd}$ is an ambient isotopy invariant. When two oriented knot diagrams portray the same knot type, i.e., they are {\it ambient isotopic}, one can add a certain number of {\itshape self-writhes} to one of them, such that they have the same writhe and Whitney degree. 

Here adding self-writhes to an oriented knot diagram means adding a certain number of the following types of tangles to it.
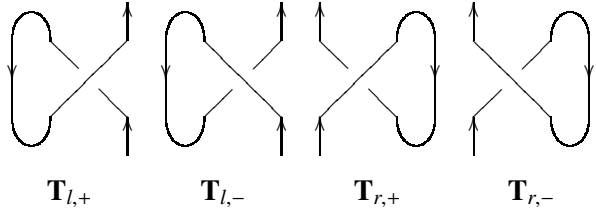
\[
\xy 0;/r.12pc/:
(-10,-10);(-30,10)**\crv{};\POS?(.5)*{\hole}="x";
(-30,-10);"x"**\crv{};(-10,10);"x"**\crv{};
(-30,-10);(-40,-10)**\crv{(-30,-20)&(-40,-20)};
(-30,10);(-40,10)**\crv{(-30,20)&(-40,20)};
(-40,-10);(-40,10)**\dir{-};?(0.5)*\dir{<};
(-10,-20);(-10,-10)**\dir{-};?(1)*\dir{>};
(-10,10);(-10,20)**\dir{-};?(1)*\dir{>};
(-50,10);(-70,-10)**\crv{};\POS?(.5)*{\hole}="x";
(-50,-10);"x"**\crv{};(-70,10);"x"**\crv{};
(-70,10);(-80,10)**\crv{(-70,20)&(-80,20)};
(-70,-10);(-80,-10)**\crv{(-70,-20)&(-80,-20)};
(-80,10);(-80,-10)**\dir{-};?(0.5)*\dir{>};
(-50,-20);(-50,-10)**\dir{-};?(1)*\dir{>};
(-50,10);(-50,20)**\dir{-};?(1)*\dir{>};
(0,-10);(20,10)**\crv{};\POS?(.5)*{\hole}="x";
(20,-10);"x"**\crv{};(0,10);"x"**\crv{};
(20,-10);(30,-10)**\crv{(20,-20)&(30,-20)};
(20,10);(30,10)**\crv{(20,20)&(30,20)};
(30,-10);(30,10)**\dir{-};?(0.5)*\dir{<};
(0,-20);(0,-10)**\dir{-};?(1)*\dir{>};
(0,10);(0,20)**\dir{-};?(1)*\dir{>};
(40,10);(60,-10)**\crv{};\POS?(.5)*{\hole}="x";
(40,-10);"x"**\crv{};(60,10);"x"**\crv{};
(60,10);(70,10)**\crv{(60,20)&(70,20)};
(60,-10);(70,-10)**\crv{(60,-20)&(70,-20)};
(70,10);(70,-10)**\dir{-};?(0.5)*\dir{>};
(40,-20);(40,-10)**\dir{-};?(1)*\dir{>};
(40,10);(40,20)**\dir{-};?(1)*\dir{>};
(-65,-30)*{\mb{T}_{l,+}};(-25,-30)*{\mb{T}_{l,-}};
(15,-30)*{\mb{T}_{r,+}};(55,-30)*{\mb{T}_{r,-}};
\endxy
\]
The following important result is needed for us (see \cite[Thm. 4.2 (a)]{RW}).

{\it Denote by {\upshape$\mb{Curl}$} the collection of all 1-1 tangle diagrams constituted by the four types of oriented 1-1 tangles above. Then for any {\upshape $\mb{T}\in\mb{Curl}$}, the specialization of its formal product to $H$ is
{\upshape\begin{equation}\label{T}
\mb{w}_{H}(\mb{T})=\ms{h}^{\scriptsize-(\ms{Wr}(\mb{T})+\ms{Wd}(\mb{T}))/2}
u^{\scriptsize-\ms{Wr}(\mb{T})},
\end{equation}}
where {\upshape $\ms{h}=u^{-1}s(u)$}. Note that for any {\upshape$\mb{T}\in\mb{Curl}$, $\ms{Wr}(\mb{T})+\ms{Wd}(\mb{T})$} is always even.}

Finally, we are in the position to do the {\it writhe-normalization}.
\begin{proposition}\label{wnor}
Given a ribbon Hopf algebra $(H, R, v)$ and its finite dimensional irreducible representation $(\pi,M)$, the TOQA structure
$(H,\, R,\, \mi_{H},\, s^{-2},\, u^{-1}v)$
provides an ambient isotopy invariant of oriented knots (in fact even for oriented links)
\begin{equation}\label{inv1}
\overline{\mb{\upshape Inv}}_{H,M}=c^{\scriptsize\ms{\upshape Wr}}\mb{\upshape Inv}_{H,M}
\end{equation}
with scalar $c\in k$ given by the action of the ribbon element $v$ on $M$.
\end{proposition}

\begin{proof}
For any two knot diagrams $\mb{\upshape K}_1$ and $\mb{\upshape K}_2$, which portray the same knot type, by adding a certain tangle diagram $\mb{\upshape T}\in\mb{\upshape Curl}$ to $\mb{\upshape K}_2$, one gets another knot diagram $\mb{\upshape K}_2'$ regular isotopic to $\mb{\upshape K}_1$, where $\ms{\upshape Wr}(\mb{\upshape T})=\ms{\upshape Wr}(\mb{\upshape K}_1)-\ms{\upshape Wr}(\mb{\upshape K}_2),
~\ms{\upshape Wd}(\mb{\upshape T})=\ms{\upshape Wd}(\mb{\upshape K}_1)-\ms{\upshape Wd}(\mb{\upshape K}_2)$.
Now
by the construction mentioned in Section 4, we have
\[
\begin{split}
G^{\scriptsize\ms{Wd}(\mb{\upshape K}_1)}\mb{w}_{H}(\mb{\upshape K}_1)
=G^{\scriptsize\ms{Wd}(\mb{\upshape K}_2')}\mb{w}_{H}(\mb{\upshape K}_2')
 &=G^{\scriptsize\ms{Wd}(\mb{T})}\mb{w}_{H}(\mb{T}) G^{\scriptsize\ms{Wd}(\mb{\upshape K}_2)}\mb{w}_{H}(\mb{\upshape K}_2)\\
&=\lb\ms{h}u^2\rb^{-\scriptsize(\ms{Wr}(\mb{T})+\ms{Wd}(\mb{T}))/2}v^{\scriptsize\ms{Wd}(\mb{T})}G^{\scriptsize\ms{Wd}(\mb{\upshape K}_2)}\mb{w}_{H}(\mb{\upshape K}_2)\\
&=v^{-\scriptsize\ms{Wr}(\mb{T})}G^{\scriptsize\ms{Wd}(\mb{\upshape K}_2)}\mb{w}_{H}(\mb{\upshape K}_2)\\
&=v^{\scriptsize\ms{\upshape Wr}(\mb{\upshape K}_2)-\ms{\upshape Wr}(\mb{\upshape K}_1)}G^{\scriptsize\ms{Wd}(\mb{\upshape K}_2)}\mb{w}_{H}(\mb{\upshape K}_2),
\end{split}
\]
where we use (\ref{T}) and the fact that $G$ and $\mb{w}_{H}(\mb{T})$, $\mb{\upshape T}\in\mb{\upshape Curl}$, commute.

By Schur's lemma, the ribbon element $v$ acts by multiplying a scalar $c\in k$. Hence,
\begin{equation*}
\mb{\upshape Inv}_{H,M}(\mb{\upshape K}_1)=
c^{\scriptsize\ms{\upshape Wr}(\mb{\upshape K}_2)-\ms{\upshape Wr}(\mb{\upshape K}_1)}
\mb{\upshape Inv}_{H,M}(\mb{\upshape K}_2)
\end{equation*}
by \eqref{eq:6}, which also implies that $\overline{\mb{\upshape Inv}}_{H,M}$ defined by\eqref{inv1} is an ambient isotopy invariant.
\end{proof}

\begin{remark}
From \cite[Thm. 4]{Rad2}, we know that when turning to representations of the ribbon Hopf algebra $(H, R, v)$, the construction by bead-sliding has a close relation with the classical construction of quantum link invariants. In general, it coincides with the one constructed from a ribbon category, due to Reshetikhin-Turaev \cite{RT1}. The regular invariant $\mb{\upshape Inv}_{H,M}$ serves as an invariant of oriented framed links, and when normalized by a ``twist factor'', it becomes the one, $\overline{\mb{\upshape Inv}}_{H,M}$, which does not depend on framing. One can also check the detail from \cite[Thm. 1.2]{MPS}.

We also note that the initial value $\overline{\mb{Inv}}_{H,M}(\mb{O})=\chi_M(u^{-1}v)$, for the unknot $\mb{O}$, is just the usual {\itshape quantum dimension} of $(\pi,M)$.
\end{remark}

\section{From the Drinfel'd double $D(\bar{\mathscr{A}})$ to knot invariants}

Now we are in the position to apply the construction introduced in the previous section to our target $D(\bar{\mathscr{A}})$.
\subsection{Representations of the Drinfel'd double $D(\bar{\mathscr{A}})$}
Throughout the rest of paper, we will concentrate on the TOQA
\[\lb D(\bar{\mathscr{A}}),\, \mathcal
{R},\, \mbox{Id}_{D(\bar{\mathscr{A}})},\, S^2\ot s^{-2},\, u^{-1}v\rb.\]
For convenience, we always assume that $D(\bar{\mathscr{A}})$ is a ribbon Hopf algebra throughout the rest of this section, which requires $\ell$ to be odd by Thm. \ref{ribb}. In the later computation, when $q$ is a primitive $\ell$-th root of unity, we see that $q^{(\ell\pm1)i/2}=(-1)^iq^{\pm i/2}
,\, i\in\mathbb{Z}$.

From \cite[Cor. 1]{Rad}, we see that for a special kind of graded Hopf algebras, the author had totally portrayed their simple Yetter-Drinfel'd modules. The description is the following.

{\itshape Let $H=\bigoplus\nolimits_{i=0}^\infty H_{(i)}$ be a ($\mathbb{Z}$-)graded Hopf algebra over an algebraically closed field $k$. Suppose that $H_{(0)}=kG$, where $G$ is a finite abelian group, $H_{(0)}\simeq k\times\cdots\times k$ as algebras, and  $H_{(i)}=H_{(i+1)}=\cdots=0$ for some $i>0$. Then
\[(\be,g)\mapsto[H_{\be,g}]\]
is a bijection between the Cartesian product of sets of group-like elements
$G(H^\circ)\times G=G(H^\circ)\times G(H)$ and the set of isomorphism classes of the simple Yetter-Drinfel'd  $H$-modules, where $H^\circ$ is the dual Hopf algebra of $H$.}

\medskip
Next we give the explicit definition of $H_{\be,g}\in{_H\mathcal {Y}\mathcal {D}^H}$. First one can easily check that for any $\be\in G(H^\circ)$, $H_\be=(H,\bullet_\be,\De)\in{_H\mathcal {Y}\mathcal {D}^H}$, where the module action $\bullet_\be$ is defined as
\[h\bullet_\be a=(h_{(2)}\leftharpoonup\mk(\be))\,a\,\vs(h_{(1)}), \quad\forall\, h, a\in H.\]
Under such action, we know that $H_\be$ inherits the graded structure of $H$, and becomes a graded left $H$-module. Now take
\[H_{\be,g}=H\bullet_\be g,\]
then it is a simple Yetter-Drinfel'd $H$-submodule of $H_\be$.

\medskip
In the case when $H=\bar{\mathscr{A}}_q(n)$, take $\bar{\mathscr{A}}_{(i)}=\ms{span}_k\bigl\{\,x^\ga K(\al)\mid \ga\in\La_+^{\leq\ka}, \, |\ga|=i\,\bigr\}$, then
 $\bar{\mathscr{A}}_q(n)=\bigoplus\nolimits_{i=0}^\infty \bar{\mathscr{A}}_{(i)}$ is a graded Hopf algebra, with
 the coradical $\bar{\mathscr{A}}_{(0)}=kG\lb\bar{\mathscr{A}}_q(n)\rb\cong k\mathbb{Z}_\ell^n$, by Proposition \ref{m1}, i.e., $\bar{\mathscr{A}}_q(n)$ is pointed.
 Moreover, the Jacobson radical of $\bar{\mathscr{A}}_q(n)$ is
 $\bigoplus\nolimits_{i=1}^\infty \bar{\mathscr{A}}_{(i)}$, with $\bar{\mathscr{A}}_{(i)}=0$ for all $i>n(\ell-1)$.

Since $\bar{\mathscr{A}}_q(n)$ is finite-dimensional, its dual is just $\bar{\mathscr{A}}_q(n)^*$.
Hence by \cite[Cor. 1]{Rad} and Proposition \ref{m2}, we have
\begin{theorem}\label{dmod}
Any left simple $D(\bar{\mathscr{A}})$-module has the form  $\lb\bar{\mathscr{A}}_q(n)\rb_{\mk(\al),K(\be)}$, $\al, \be\in\La_+^{\leq\ka}$, simply denoted by $\bar{\mathscr{A}}_{\al,\be}$.
\end{theorem}
Let us start to study these simple left $D(\bar{\mathscr{A}})$-modules.

First, we should clarify the module action of $\lb\bar{\mathscr{A}}_{\al,\be},\cdot\rb$  as a left $D(\bar{\mathscr{A}})$-module:
\begin{equation}
\label{eq10}
\begin{split}
x^\ga K(\xi)\bullet_{\mk(\al)}K(\be)&=\sum\limits_{\ze\in\La_+^{\leq\ga}}{\ga\choose\ze}
q^{-(\ga-\ze)*\ze}\lb x^\ze K(\xi)\leftharpoonup \mk(\al)\rb K(\be)\,s^{-1}\lb x^{\ga-\ze}K(\ze+\xi)\rb\\
&=\mu(\ga,\xi:\al,\be)x^\ga K(\be-\ga),
\end{split}
\end{equation}
where we set
\[\mu(\ga,\xi:\al,\be)=\sum\limits_{\ze\in\La_+^{\leq\ga}}(-1)^{|\ga-\ze|}{\ga\choose\ze}
q^{\tfrac{1}{2}\lan 2\be-\ga-\ze+\pmb{1},\ga-\ze\ran+\lan \al,\xi+\ze\ran}\te(\be+\ga-\ze,\ga-\ze).\]
In addition,
\begin{equation}
\label{eq11}
\begin{split}
\mx^\ga\mk(\om)\rightharpoonup x^\eta K(\be-\eta)&=\sum\limits_{\xi\in\La_+^{\leq\eta}}{\eta\choose\xi}
q^{-(\eta-\xi)*\xi}\mx^\ga\mk(\om)\lb x^\xi K(\be-\eta)\rb x^{\eta-\xi}K(\xi+\be-\eta)\\
&={\eta\choose\ga}q^{-(\eta-\ga)*\ga+\lan\om,\be-\eta\ran}x^{\eta-\ga}K(\be+\ga-\eta).
\end{split}
\end{equation}
Note that in Eq. (\ref{eq11}), if $\ga\nleq\eta$, $\displaystyle{{\eta\choose\ga}=0}$, so the corresponding result should be interpreted as $0$.

Combining Eq. (\ref{eq10}) with Eq. (\ref{eq11}), we obtain the formula of the module action of  $(\bar{\mathscr{A}}_{\al,\be},\cdot)$ as a left $D(\bar{\mathscr{A}})$-module
\begin{equation}
\begin{split}
\lb\mx^\ga\mk(\om)\ot x^\eta K(\xi)\rb\cdot K(\be)&=\mu(\eta,\xi:\al,\be)\lb\mx^\ga\mk(\om)\rightharpoonup x^\eta K(\be-\eta)\rb\\
&={\eta\choose\ga}q^{-(\eta-\ga)*\ga+\lan\om,\be-\eta\ran}\mu(\eta,\xi:\al,\be)x^{\eta-\ga}K(\be+\ga-\eta),
\end{split}
\end{equation}
where $\ga,\eta,\om,\xi\in\La_+^{\leq\ka}$.

From Eq. (\ref{eq10}), we can see that $\bar{\mathscr{A}}_{\al,\be}$ lies in the subspace of $\bar{\mathscr{A}}_q(n)$ expanded by $\bigl\{\,x^\ga K(\be-\ga)\mid \ga\in\La_+^{\leq\ka}\,\bigr\}$. In order to find a $k$-basis of $\bar{\mathscr{A}}_{\al,\be}$, the key is to confirm when the coefficient $\mu(\ga,\xi:\al,\be)$ does not vanish.

To this end, we consider
\begin{equation*}
\begin{split}
x_i\bullet_{\mk(\al)}x^\ga K(\be-\ga)&=\lb x_i\leftharpoonup\mk(\al)\rb x^\ga K(\be-\ga)K(\ep_i)^{-1}\\
&\quad+\lb 1\leftharpoonup\mk(\al)\rb x^\ga K(\be-\ga)\lb -x_iK(\ep_i)^{-1}\rb\\
&=q^{\ep_i*\ga}\lb q^{\lan\al,\ep_i\ran}-\te(\be,\ep_i)q^{\lan\be-\ga,\ep_i\ran}\rb x^{\ga+\ep_i}K(\be-\ga-\ep_i).
\end{split}
\end{equation*}

Now let $x_i\bullet_{\mk(\al)}x^\ga K(\be-\ga)=0$, for $1\le  i\le n$. If $x^{\ga+\ep_i}\neq0$, then $q^{\lan\al,\ep_i\ran}-\te(\be,\ep_i)q^{\lan\be-\ga,\ep_i\ran}=0$, i.e., $q^{\lan\al+\ga,\ep_i\ran}=q^{\al_i+\ga_i}=\te(\be,\ep_i)q^{\lan\be,\ep_i\ran}=q^{-\tilde{\be}_i+2\be_i}$.
Since for $\al\in\La$,
there exists a unique $[\al]=([\al_1],\cdots,[\al_n])\in\La_+^{\leq\ka}$,
such that $[\al]\equiv\al~(\ms{mod}~\iota)$,
we can solve that
\[\ga=[-\tilde{\be}+2\be-\al],\]
and set
\[\ka(\al,\be)=[-\tilde{\be}+2\be-\al], \ \al,\; \be\in\La_+^{\leq\ka}.\]

To describe the coefficient $\mu(\ga,\xi:\al,\be)$, we need to extract a linearly ordered subset from $\La_+^{\leq\ga}$.
Arranging that subset from left to right, we have
\[0,~\ep_n,~2\ep_n,~\ldots,~\ga_n\ep_n,~\ep_{n-1}+\ga_n\ep_n,~2\ep_{n-1}+\ga_n\ep_n,~\ldots,~
\ga_{n-1}\ep_{n-1}+\ga_n\ep_n,~\ldots,~\ga.\]
Denote the subset by $\underline{\ga}$.

For $\eta\in\underline{\ga}$, take $i(\eta)\in\{1,\cdots,n\}$ to be the unique subscript such that $\eta+\ep_{i(\eta)}$ is the successor of $\eta$
in terms of the linear order above.
 Then as $x^\ga=x_1^{\ga_1}\cdots x_n^{\ga_n}$,
 we can compute $x^\ga K(\xi)\bullet_{\mk(\al)}K(\be)=q^{\lan\al,\xi\ran}x^\ga\bullet_{\mk(\al)}K(\be)$ in the order of $\underline{\ga}$, and get
\begin{equation}\label{eq12}
x^\ga K(\xi)\bullet_{\mk(\al)}K(\be)
=q^{\lan\al,\xi\ran}\prod\limits_{\eta\in\underline{\ga}\backslash\{\ga\}}q^{\ep_{i(\eta)}*\eta}
\lb q^{\lan\al,\ep_{i(\eta)}\ran}-\te(\be,\ep_{i(\eta)})q^{\lan\be-\eta,\ep_{i(\eta)}\ran}\rb x^\ga K(\be-\ga).
\end{equation}

Comparing Eq. (\ref{eq10}) with Eq. (\ref{eq12}), we have
\[\mu(\ga,\xi:\al,\be)=
\begin{cases}
\displaystyle{q^{\lan\al,\xi\ran}\prod\limits_{\eta\in\underline{\ga}\backslash\{\ga\}}q^{\ep_{i(\eta)}*\eta}
\lb q^{\lan\al,\ep_{i(\eta)}\ran}-\te(\be,\ep_{i(\eta)})q^{\lan\be-\eta,\ep_{i(\eta)}\ran}\rb},
& \ga\in\La_+^{\leq\ka(\al,\be)},\\
0, & \mbox{otherwise}.
\end{cases}
\]

Now as $\bar{\mathscr{A}}_{\al,\be}=\bar{\mathscr{A}}_q(n)\bullet_{\mk(\al)}K(\be)$, we have
\begin{proposition}\label{basis}
$\bigl\{\,x^\ga K(\be-\ga)\mid \ga\in\La_+^{\leq\ka(\al,\be)}\,\bigr\}$
is a $k$-basis of $\bar{\mathscr{A}}_{\al,\be}\,$, for $\,\al, \, \be\in\La_+^{\leq\ka}$.
\end{proposition}

For later use, we also give a criterion when the irreducible module $\bar{\mathscr{A}}_{\al,\be}$ 
is self-dual.
\begin{theorem}\label{sdual}
When $\ell$ is odd, the irreducible module $\bar{\mathscr{A}}_{\al,\be}\,(\al,\be\in\La_+^{\leq\ka})$ of $D(\bar{\mathscr{A}}_n(q))$ is self-dual if and only if there exists a unique tuple $r=(r_1,\dots,r_n)\in\La_+^{\leq\ka}$ such that
\[\be=[(i_0r_1,\dots,i_0r_n)],~
\al=[-\tilde \be],\]
where $i_0=(\ell+1)/2$, and we interpret the tuples above modulo $\ell$. In particular, if it is the case, then \upshape{dim}$(\bar{\mathscr{A}}_{\al,\be})=(r_1+1)\cdots(r_n+1)$ as $\ka(\al,\be)=r$, and any self-dual irreducible module is uniquely determined by this tuple $\ka(\al,\be)\in\La_+^{\leq\ka}$.
\end{theorem}
\begin{proof}
Since $(i_0,\ell)=1$, given $\be\in\La_+^{\leq\ka}$ there exists a unique $(r_1,\dots,r_n)\in\La_+^{\leq\ka}$ such that $\be=[(i_0r_1,\dots,i_0r_n)]$. Now what we need to compute is the action of $K(\xi),~\mk(\om),~\xi,\om\in\La_+^{\leq\ka}$ on $\bar{\mathscr{A}}_{\al,\be}$. First, we have
\[\begin{array}{l}
K(\xi)\bullet_{\mk(\al)}x^\ga K(\be-\ga)=q^{\lan\al+\tilde\ga,\xi\ran}x^\ga K(\be-\ga),\\
\mk(\om)\rightharpoonup x^\ga K(\be-\ga)=q^{\lan\om,\be-\ga\ran}x^\ga K(\be-\ga)
\end{array}\]
for any $\ga\in\La_+^{\leq\ka(\al,\be)}$.
In order to get a self-dual module, the eigenvalues of each such action in terms of the eigenvectors from Prop.~\ref{basis} should be reciprocal in pairs. More precisely, $\exists \ga,\ga'\in\La_+^{\leq\ka(\al,\be)}$ such that
\[\lan\ep_i,\be_i-\ga_i\ran=-\lan\ep_i,\be_i-\ga'_i\ran,~\lan \al_i+\tilde\ga_i,\ep_i\ran=-\lan\al_i+\tilde\ga'_i,\ep_i\ran,~i=1,\dots,n.\]
And that gives the equality $\al=[-\tilde\be]$. Conversely, if $\al=[-\tilde \be]$, then $\ka(\al,\be)=[2\be]$ and $(x^\ga K(\be-\ga),x^{2\be-\ga}K(-\be+\ga)),\ga\in\La_+^{\leq[2\be]}$ are those pairs as desired.
\end{proof}

\subsection{Knot invariants from the Drinfel'd double $D(\bar{\mathscr{A}})$}
Up to now, all simple left $D(\bar{\mathscr{A}})$-modules have been clarified.
Let us turn to the solution $R_{\bar{\mathscr{A}}_{\al,\be}}$ of the braid equation associated with $\bar{\mathscr{A}}_{\al,\be}$.
Abbreviate $R_{\bar{\mathscr{A}}_{\al,\be}}$ to $R_{\al,\be}$.
To show that the modification (\ref{eq:6}) works well in the case of $D(\bar{\mathscr{A}})$, we will compute some low rank cases below. With no confusion, let us abbreviate $\mb{Inv}_{H,M},\, \overline{\mb{Inv}}_{H,M}$ to $\mb{Inv},\, \overline{\mb{Inv}}$, respectively.

\medskip
\begin{example}\label{ex1}
The Taft algebra ($n=1$). Its simplest representation is the two dimensional irreducible one, $\bar{\mathscr{A}}_{(i-1)\ep_1,i\ep_1}$, $(1\le i\le\ell)$ with $\bar{\mathscr{A}}_{(\ell-1)\ep_1,0}$ for $i=\ell$. Since $\ka((i{-}1)\ep_1,i\ep_1)=\ep_1$,  it has a $k$-basis
$\{\,K(i\ep_1),\ x_1K((i{-}1)\ep_1)\,\}$.
By (\ref{eq0}), we have
\begin{align*}
R_{(i-1)\ep_1,i\ep_1}&(K(i\ep_1)\ot K(i\ep_1))=K(i\ep_1)\ot
\lb K(i\ep_1)\bullet_{\mk((i-1)\ep_1)}K(i\ep_1)\rb \\
&=q^{i(i-1)}K(i\ep_1)\ot K(i\ep_1),\\
R_{(i-1)\ep_1,i\ep_1}&(K(i\ep_1)\ot x_1K((i-1)\ep_1))\\
&=x_1K((i-1)\ep_1)\ot\lb K((i-1)\ep_1)\bullet_{\mk((i-1)\ep_1)}K(i\ep_1)\rb\\
&\quad+K(i\ep_1)\ot\lb x_1K((i-1)\ep_1)\bullet_{\mk((i-1)\ep_1)}K(i\ep_1)\rb\\
&=q^{(i-1)^2}(x_1K((i-1)\ep_1)\ot K(i\ep_1))\\
&\quad+\lb q^{i(i-1)}-q^{i(i-1)+1}\rb (K(i\ep_1)\ot x_1K((i-1)\ep_1)),\\
R_{(i-1)\ep_1,i\ep_1}&(x_1K((i-1)\ep_1)\ot K(i\ep_1))=K(i\ep_1)\ot \lb K(i\ep_1)\bullet_{\mk((i-1)\ep_1)}x_1K((i-1)\ep_1\rb\\
&=q^{i^2}K(i\ep_1)\ot x_1K((i-1)\ep_1),\\
R_{(i-1)\ep_1,i\ep_1}&(x_1K((i-1)\ep_1)\ot x_1K((i-1)\ep_1))\\
&=x_1K((i-1)\ep_1)\ot\lb K((i-1)\ep_1)\bullet_{\mk((i-1)\ep_1)}x_1K((i-1)\ep_1)\rb\\
&\quad+K(i\ep_1)\ot\lb x_1K((i-1)\ep_1)\bullet_{\mk((i-1)\ep_1)}x_1K((i-1)\ep_1)\rb\\
&=q^{i(i-1)}x_1K((i-1)\ep_1)\ot x_1K((i-1)\ep_1).
\end{align*}
Therefore, under the given $k$-basis
\begin{gather*}
\Big\{\,K(i\ep_1)\ot K(i\ep_1),\ K(i\ep_1)\ot x_1K((i{-}1)\ep_1),\\
x_1K((i{-}1)\ep_1)\ot K(i\ep_1), \ x_1K((i{-}1)\ep_1)\ot x_1K((i{-}1)\ep_1)\,\Big\}
\end{gather*}
of $\bar{\mathscr{A}}_{(i-1)\ep_1,i\ep_1}\ot\bar{\mathscr{A}}_{(i-1)\ep_1,i\ep_1}$,
$R_{(i-1)\ep_1,i\ep_1}$ is presented by
\[
\begin{pmatrix}
q^{i(i-1)} & 0 & 0 & 0\\
0 & q^{i(i-1)}(1-q) & q^{i^2} & 0\\
0 & q^{(i-1)^2} & 0 & 0\\
0 & 0 & 0 & q^{i(i-1)}
\end{pmatrix}.
\]

In particular, choosing $i=i_0:=(\ell{+}1)/2$, we get the self-dual irreducible one, $\bar{\mathscr{A}}_{(i_0-1)\ep_1,i_0\ep_1}$, with the following symmetric braiding $R$-matrix, as $q^{(i_0-1)^2}=q^{i_0^2}$ when $\ell$ is odd.
\[R_{(i_0-1)\ep_1,i_0\ep_1}=
\begin{pmatrix}
q^{-1/4} & 0 & 0 & 0\\
0 & q^{-1/4}-q^{3/4} & q^{1/4} & 0\\
0 & q^{1/4} & 0 & 0\\
0 & 0 & 0 & q^{-1/4}
\end{pmatrix},
\]
with the minimal polynomial
\begin{equation}\label{poly1}
\ms{minpoly}(R_{(i_0-1)\ep_1,i_0\ep_1})=\lb x-q^{-1/4}\rb\lb x+q^{3/4}\rb,
\end{equation}
where we abuse to denote $q^{1/4}:=q^{(\ell+1)^2/4}$ as a fourth root of $q$, and $q^{i/4}:=q^{(\ell+1)^2i/4}$ for any $i\in \mathbb Z$. In particular, $q^{(\ell^2-1)/4}=q^{-(\ell+1)^2/4}=q^{-1/4}$.

Under the basis
$\{\,K(i_0\ep_1), \ x_1K((i_0{-}1)\ep_1)\,\}$,
 $u$ is presented by
${\rm diag}\{q^{5/4}, q^{1/4}\}$,
while $v$ is presented by $q^{3/4}\mbox{Id}_{\bar{\mathscr{A}}_{(i_0-1)\ep_1,i_0\ep_1}}$, which we obtain from (\ref{rib}).

As
\[q^{-1/4}R_{(i_0-1)\ep_1,i_0\ep_1}-q^{1/4}R_{(i_0-1)\ep_1,i_0\ep_1}^{-1}=
\lb q^{-1/2}-q^{1/2}\rb\mbox{Id}_{\bar{\mathscr{A}}_{(i_0-1)\ep_1,i_0\ep_1}},\]
we see that
\[q^{-1}\overline{\mb{Inv}}
\lb\xy
0;/r.1pc/:\ar(-6,-6)*{};(6,6)*{};\POS?(.5)*{\hole}="x";
    (-6,6)*{};"x"+(-1.3,1.3)**\crv{};
    \ar(6,-6)*{};"x"+(1.3,-1.3);
\endxy\rb
-q\overline{\mb{Inv}}
\lb\xy
0;/r.1pc/:\ar(6,-6)*{};(-6,6)*{};\POS?(.5)*{\hole}="x";
    (6,6)*{};"x"+(1.3,1.3)**\crv{};
    \ar(-6,-6)*{};"x"+(-1.3,-1.3);
\endxy\rb
=\lb q^{-1/2}-q^{1/2}\rb\overline{\mb{Inv}}
\lb\xy
0;/r.1pc/:(-6,6)*{};(-6,-6)*{}**\crv{(4,0)}?(1)*\dir{>};
(6,6)*{};(6,-6)*{}**\crv{(-4,0)}?(1)*\dir{>};
\endxy\rb.\]
Since $\overline{\mb{Inv}}(\mb{O})=\chi(u^{-1}v)=q^{-\tfrac{1}{2}}+q^{\tfrac{1}{2}}$, it means that $\overline{\mb{Inv}}$ constructed here is nothing but the specialization of the {\itshape Jones polynomial} $V_L(t)$ at
$t=q$, where $V_L(t)$ satisfies the following skein relations:
\[\begin{array}{l}
t^{-1}V_{\left\lan\xy
0;/r.08pc/:\ar(-6,-6)*{};(6,6)*{};\POS?(.5)*{\hole}="x";
    (-6,6)*{};"x"+(-1.3,1.3)**\crv{};
    \ar(6,-6)*{};"x"+(1.3,-1.3);
\endxy\right\ran}(t)
-tV_{\left\lan
\xy 0;/r.08pc/:\ar(6,-6)*{};(-6,6)*{};\POS?(.5)*{\hole}="x";
    (6,6)*{};"x"+(1.3,1.3)**\crv{};
    \ar(-6,-6)*{};"x"+(-1.3,-1.3);
\endxy\right\ran}(t)
=\lb t^{-1/2}-t^{1/2}\rb V_{\left\lan\xy
0;/r.08pc/:(-6,6)*{};(-6,-6)*{}**\crv{(4,0)}?(1)*\dir{>};
(6,6)*{};(6,-6)*{}**\crv{(-4,0)}?(1)*\dir{>};
\endxy\right\ran}(t),\\
V_{ L~\sqcup~\xy
0;/r.08pc/:(-4,0);(4,0)**\crv{(-4,6)&(4,6)};(-4,0);(4,0)**\crv{(-4,-6)&(4,-6)};
\endxy\,}(t)=\lb t^{-1/2}+t^{1/2}\rb V_L(t),~V_{\xy
0;/r.08pc/:(-4,0);(4,0)**\crv{(-4,6)&(4,6)};(-4,0);(4,0)**\crv{(-4,-6)&(4,-6)};
\endxy}(t)=t^{-1/2}+t^{1/2}.
\end{array}\]

\end{example}

\begin{remark}
(1) The version of $V_L(t)$ used here differs from the usual one, \cite[1.2]{Oht}, by a multiple $(-1)^{\sharp L}\lb -t^{-1/2}-t^{1/2}\rb$, where $\sharp L$ is the number of components of $L$. Note that $R_{(i_0-1)\ep_1,i_0\ep_1}$ is related to the {\itshape Kauffman bracket} $\Big\lan\cdot\Big\ran(A)$, a regular invariant of unoriented links defined by
\[\hspace{-3em}
\xy
0;/r.1pc/:
(-6,6)*{};(-1.5,1.5)*{}**\dir{-};
(6,-6)*{};(1.5,-1.5)*{}**\dir{-};
(-6,-6)*{};(6,6)*{}**\dir{-};
\endxy=
A~\xy
0;/r.1pc/:
(-6,-6)*{};(-6,6)*{}**\crv{(0,0)};
(6,6)*{};(6,-6)*{}**\crv{(0,0)};
\endxy+
A^{-1}~\xy
0;/r.1pc/:
(-6,-6)*{};(6,-6)*{}**\crv{(0,0)};
(6,6)*{};(-6,6)*{}**\crv{(0,0)};
\endxy~,\]
\[
\xy
0;/r.1pc/:
(-6,6)*{};(-1.5,1.5)*{}**\dir{-};
(5,-5)*{};(1.5,-1.5)*{}**\dir{-};
(-6,-6)*{};(5,5)*{}**\dir{-};
(5,5)*{};(5,-5)*{}**\crv{(10,10)&(10,-10)};
\endxy =-A^3~\xy
0;/r.1pc/:
(-6,6)*{};(-6,-6)*{}**\crv{(0,0)};
\endxy~,~
\xy
0;/r.1pc/:
(-6,-6)*{};(-1.5,-1.5)*{}**\dir{-};
(5,5)*{};(1.5,1.5)*{}**\dir{-};
(-6,6)*{};(5,-5)*{}**\dir{-};
(5,5)*{};(5,-5)*{}**\crv{(10,10)&(10,-10)};
\endxy=-A^{-3}~\xy
0;/r.1pc/:
(-6,6)*{};(-6,-6)*{}**\crv{(0,0)};
\endxy~,
\]
and $\left\lan L\sqcup\xy
0;/r.1pc/:(-4,0);(4,0)**\crv{(-4,6)&(4,6)};(-4,0);(4,0)**\crv{(-4,-6)&(4,-6)};
\endxy\right\ran(A)=\lb -A^2-A^{-2}\rb\Big\lan L\Big\ran(A),~\Big\lan \emptyset\Big\ran(A)=1$. Also,
\[V_L(t)=\left((-A)^{-3\scriptsize\ms{Wr}(L)}\Big\lan L\Big\ran(A)\right)_{A^2=-t^{-1/2}}.\]

As $v$ acts by multiplying $q^{3/4}$ rather than $-q^{3/4}$,
one can check that
\[\mb{Inv}(L)
=(-1)^{\scriptsize\ms{Wr}(L)}\Big\lan L\Big\ran\lb \sqrt{-1}q^{-1/4}\rb=\Big\lan L\Big\ran\lb -\sqrt{-1}q^{-1/4}\rb,\]
becomes a regular invariant of unoriented links, as $(-1)^{\scriptsize\ms{Wr}(L)}$ is independent of the orientation of $L$. It is just the quantum $(\mathfrak{sl}_2,V_2)$ invariant with $q^{1/4}$ replacing by $q^{-1/4}$ \cite[Prop. 4.19. and Th H.1.]{Oht}.

\smallskip\noindent
(2) If we adopt the version of $U_q(\mathfrak{s}\mathfrak{l}_2)$ from the appendix in \cite{Mon}, then the Taft algebra $\bar{\mathscr{A}}_q(1)=T_\ell$ is isomorphic to the positive part $\overline{U}_{q'}(\mathfrak{b})$ of the restricted $\overline{U}_{q'}(\mathfrak{s}\mathfrak{l}_2)$ with $q'=q^{1/4}$, where $q$ is assumed to have a 4-th primitive root of unity. Therefore,
\[D(\bar{\mathscr{A}}_q(1))/\lb K(\ep_1)-\mk(\ep_1)^{-1}\rb\cong \overline{U}_{q'}(\mathfrak{s}\mathfrak{l}_2).\]
Note that $\bar{\mathscr{A}}_{(i_0-1)\ep_1,i_0\ep_1}$ is the only two dimensional {\itshape self-dual} irreducible module of $D(\bar{\mathscr{A}}_q(1))$, i.e. $\bar{\mathscr{A}}_{(i_0-1)\ep_1,i_0\ep_1}\cong {\bar{\mathscr{A}}_{(i_0-1)\ep_1,i_0\ep_1}}^*$, and can be factorized as the fundamental representation of $U_{q'}(\mathfrak{s}\mathfrak{l}_2)$, which fits the fact that the reconstruction of the Jones polynomial as quantum invariants is derived from the fundamental representation of $U_q(\mathfrak{s}\mathfrak{l}_2)$.
\end{remark}

\medskip
\begin{example}\label{ex2}
The $2$-rank Taft algebra. Consider $\bar{\mathscr{A}}_{(\ell-1)\ep_2,i_0(\ep_1+\ep_2)}$, the four dimensional self-dual simple Yetter-Drinfel'd module with $\ka((\ell-1)\ep_2,i_0(\ep_1+\ep_2))=\ep_1+\ep_2$. It has a $k$-basis
\begin{equation}\label{ba}
\big\{K(i_0(\ep_1+\ep_2)),~x_1K((i_0-1)\ep_1+i_0\ep_2),~x_2K(i_0\ep_1+(i_0-1)\ep_2),~x_1x_2K((i_0-1)(\ep_1+\ep_2))\big\}.
\end{equation}
\end{example}

Since $\bar{\mathscr{A}}_{(\ell-1)\ep_2,i_0(\ep_1+\ep_2)}$ as a left $\bar{\mathscr{A}}_q(2)$-module is graded, the annihilator
\[\mbox{\upshape Ann}_{\bar{\mathscr{A}}_q(2)}\lb\bar{\mathscr{A}}_{(\ell-1)\ep_2,i_0(\ep_1+\ep_2)}\rb\]
of $\bar{\mathscr{A}}_{(\ell-1)\ep_2,i_0(\ep_1+\ep_2)}$ is an ideal generated by
\[\left\{x^\ga~\big|~\ga\in\La_+^{\leq\ka},~\ga\neq 0,\ep_1,\ep_2,\ep_1+\ep_2\right\}.\]
In addition, $K(\al),~x_1,~x_2$, and $x_1x_2$ are successively presented by
\[
\begin{array}{cccc}
\begin{pmatrix}
q^{-\al_2} & 0 & 0 & 0\\
0 & q^{\al_1} & 0 & 0\\
0 & 0 & q^{-\al_1} & 0\\
0 & 0 & 0 & q^{\al_2}
\end{pmatrix}, &
\begin{pmatrix}
0 & 0 & 0 & 0\\
1-q & 0 & 0 & 0\\
0 & 0 & 0 & 0\\
0 & 0 & 1-q & 0
\end{pmatrix}, &
\begin{pmatrix}
0 & 0 & 0 & 0\\
0 & 0 & 0 & 0\\
q^{-1}(1-q) & 0 & 0 & 0\\
0 & 1-q & 0 & 0
\end{pmatrix}, &
\begin{pmatrix}
0 & 0 & 0 & 0\\
0 & 0 & 0 & 0\\
0 & 0 & 0 & 0\\
q^{-1}(1-q)^2 & 0 & 0 & 0
\end{pmatrix},
\end{array}
\]
with respect to the basis (\ref{ba}). It follows that the action of $K(\al)$ on $\bar{\mathscr{A}}_{(\ell-1)\ep_2,i_0(\ep_1+\ep_2)}$ is semisimple, while those of other non-zero graded elements are nilpotent.

On the other hand, take $f_{\ga,\al}\in\bar{\mathscr{A}}_q(n)^*$ dual to $x^\ga K(\al)$, and it is presented by
\[
\begin{pmatrix}
\de_{\ga,0}\de_{\al,i_0(\ep_1+\ep_2)} & \de_{\ga,\ep_1}\de_{\al,(i_0-1)\ep_1+i_0\ep_2} & \de_{\ga,\ep_2}\de_{\al,i_0\ep_1+(i_0-1)\ep_2} & \de_{\ga,\ep_1+\ep_2}\de_{\al,(i_0-1)(\ep_1+\ep_2)}\\
0 & \de_{\ga,0}\de_{\al,(i_0-1)\ep_1+i_0\ep_2} & 0 & \de_{\ga,\ep_2}\de_{\al,(i_0-1)(\ep_1+\ep_2)}\\
0 & 0 & \de_{\ga,0}\de_{\al,i_0\ep_1+(i_0-1)\ep_2} & q^{-1}\de_{\ga,\ep_1}\de_{\al,(i_0-1)(\ep_1+\ep_2)}\\
0 & 0 & 0 & \de_{\ga,0}\de_{\al,(i_0-1)(\ep_1+\ep_2)}
\end{pmatrix},
\]
with respect to the basis (\ref{ba}).

From the discussion of the third section, we know that the distinguished group-like element of $\bar{\mathscr{A}}_q(n)$ is $K(\ka)$, while the distinguished group-like element of $\bar{\mathscr{A}}_q(n)^*$ is $\tilde\mk(\pmb{1})=\mk(2\ep_2)$. They are successively presented by
${\rm diag}\{q, q^{-1}, q, q^{-1}\}$ and ${\rm diag}\{q, q, q^{-1}, q^{-1}\}$,
with respect to the basis (\ref{ba}).

  We note that $D(H)$ is unimodular for $H$ finite dimensional, i.e. the distinguished group-like element of $D(H)^*$ is trivial, by part (a) of \cite[Th. 4]{Rad1}, and $us(u)^{-1}=\al\ot g$ is the distinguished group-like element of $D(H)$ by \cite[Cor. 7]{Rad1}. In particular, the distinguished group-like element of $D(\bar{\mathscr{A}})$ is $\mk(2\ep_2)\ot K(\ka)$. We take $\ms{h}=u^{-1}s(u)=\lb\mk(2\ep_2)\ot K(\ka)\rb^{-1}$, which is presented by
${\rm diag}\{q^{-2}, 1, 1, q^2\}$
with respect to the basis (\ref{ba}).

 Since the universal $R$-matrix $\mathcal {R}$ of $D(\bar{\mathscr{A}})$ has the form as (\ref{ur}), we only need to concern those pairs $\lb x^\ga K(\al),~f_{\ga,\al}\rb,~\ga,\al\in\La_+^{\leq\ka}$ with nontrivial actions.  Checking all the matrices above, we find that $R_\pi=\pi\ot\pi(\mathcal {R})$ involves exactly nine pairs with nontrivial actions, which are
\begin{align*}
&\bigg\{\lb K(i_0\ep_1{+}(i_0{-}1)\ep_2),f_{0,i_0\ep_1{+}(i_0{-}1)\ep_2}\rb,~\lb K((i_0{-}1)(\ep_1{+}\ep_2)),f_{0,(i_0{-}1)(\ep_1{+}\ep_2)}\rb,~\lb K(i_0(\ep_1{+}\ep_2)),f_{0,i_0(\ep_1{+}\ep_2)}\rb,\\
&\lb K((i_0{-}1)\ep_1{+}i_0\ep_2),f_{0,(i_0{-}1)\ep_1{+}i_0\ep_2}\rb,~
\lb x_1K((i_0{-}1)(\ep_1{+}\ep_2)),f_{\ep_1,(i_0{-}1)(\ep_1{+}\ep_2)}\rb,\\
&\lb x_1K((i_0{-}1)\ep_1{+}i_0\ep_2),f_{\ep_1,(i_0{-}1)\ep_1{+}i_0\ep_2}\rb,~
\lb x_2K(i_0\ep_1+(i_0{-}1)\ep_2),f_{\ep_2,i_0\ep_1+(i_0{-}1)\ep_2}\rb,\\
&\lb x_2K((i_0{-}1)(\ep_1{+}\ep_2)),f_{\ep_2,(i_0{-}1)(\ep_1{+}\ep_2)}\rb,~
\lb x_1x_2K((i_0{-}1)(\ep_1{+}\ep_2)),f_{\ep_1{+}\ep_2,(i_0{-}1)(\ep_1{+}\ep_2)}\rb\bigg\}.
\end{align*}

And the pairs with nontrivial actions involved in $R_\pi^{-1}=\pi\ot\pi(\mathcal {R}^{-1})$ accordingly are
\begin{align*}
&\bigg\{\lb K((i_0{-}1)\ep_1{+}i_0\ep_2),f_{0,i_0\ep_1{+}(i_0-1)\ep_2}\rb,~\lb K(i_0(\ep_1{+}\ep_2)),f_{0,(i_0{-}1)(\ep_1{+}\ep_2)}\rb,~\lb K((i_0{-}1)(\ep_1{+}\ep_2)),f_{0,i_0(\ep_1{+}\ep_2)}\rb,\\
&\lb K(i_0\ep_1{+}(i_0{-}1)\ep_2),f_{0,(i_0{-}1)\ep_1{+}i_0\ep_2}\rb,~
\lb -x_1K((i_0{-}1)\ep_1{+}i_0\ep_2),f_{\ep_1,(i_0{-}1)(\ep_1{+}\ep_2)}\rb,\\
&\lb -q^{-1}x_1K((i_0{-}1)(\ep_1{+}\ep_2)),f_{\ep_1,(i_0{-}1)\ep_1{+}i_0\ep_2}\rb,~
\lb -x_2K((i_0{-}1)(\ep_1{+}\ep_2)),f_{\ep_2,i_0\ep_1{+}(i_0-1)\ep_2}\rb,\\
&\lb -q^{-1}x_2K(i_0\ep_1{+}(i_0{-}1)\ep_2),f_{\ep_2,(i_0{-}1)(\ep_1{+}\ep_2)}\rb,~
\lb q^{-1}x_1x_2K((i_0{-}1)(\ep_1{+}\ep_2)),f_{\ep_1{+}\ep_2,(i_0{-}1)(\ep_1{+}\ep_2)}\rb\bigg\}.
\end{align*}

Therefore, $u$ is presented by
${\rm diag}\{q^{5/2}, q^{3/2}, q^{3/2},  q^{1/2}\}$
with respect to the basis (\ref{ba}).

Note that $u^2\ms{h}=us(u)$ is a central element of $D(\bar{\mathscr{A}}_q(2))$, presented by $q^3\mi_{\bar{\mathscr{A}}_{(\ell-1)\ep_2,i_0(\ep_1+\ep_2)}}$ with respect to the basis (\ref{ba}). By Thm. \ref{ribb}, when $\ell$ is odd, there exists a unique square root, also the unique ribbon element of $D(\bar{\mathscr{A}}_q(2))$, denoted by $v$. Hence, the twist element $G=u^{-1}v$ is presented by
${\rm diag}\{q^{-1}, 1 , 1 , q\}$
with respect to the basis (\ref{ba}).

Here we write the basis (\ref{ba}) successively as $\{b_1,~b_2,~b_3,~b_4\}$. First compute the tensor product of the matrices of each pairs above with nontrivial actions, and sum them up to give $R_\pi$. Then composing it with the flip map $\tau$, we can see that the braiding $R$-matrix $R_{(\ell-1)\ep_2,i_0(\ep_1+\ep_2)}$ of the braid equation
is presented by the $16\times16$ matrix
\begin{equation}\label{rmat}
\left(
\begin{smallmatrix}{\frac {1}{\sqrt {q}}}&0&0&0&0
&0&0&0&0&0&0&0&0&0&0&0\\ 0&{\frac {1-q}{\sqrt {q}}
}&0&0&\sqrt {q}&0&0&0&0&0&0&0&0&0&0&0\\0&0&{\frac
{1-q}{\sqrt {q}}}&0&0&0&0&0&{\frac {1}{\sqrt {q}}}&0&0&0&0&0&0&0
\\ 0&0&0&{\frac { \left( 1-q \right) ^{2}}{\sqrt {q
}}}&0&0& \left( 1-q \right) \sqrt {q}&0&0& \left( 1-q \right)
\sqrt {q}&0&0&\sqrt {q}&0&0&0\\0&{\frac {1}{\sqrt {
q}}}&0&0&0&0&0&0&0&0&0&0&0&0&0&0\\0&0&0&0&0&{\frac
{1}{\sqrt {q}}}&0&0&0&0&0&0&0&0&0&0\\0&0&0& {\frac
{1-q}{\sqrt {q}}}&0&0&0&0&0&\sqrt {q}&0&0&0&0&0&0
\\ 0&0&0&0&0&0&0& {\frac {1-q}{\sqrt {q}}}&0&0&0&0&0
&\sqrt {q}&0&0\\ 0&0&\sqrt {q}&0&0&0&0&0&0&0&0&0&0&0
&0&0\\ 0&0&0& {\frac {1-q}{\sqrt {q}}}&0&0&\sqrt {q
}&0&0&0&0&0&0&0&0&0\\0&0&0&0&0&0&0&0&0&0&{\frac {1}
{\sqrt {q}}}&0&0&0&0&0\\ 0&0&0&0&0&0&0&0&0&0&0& {
\frac {1-q}{\sqrt {q}}}&0&0&{\frac {1}{\sqrt {q}}}&0
\\0&0&0&\sqrt {q}&0&0&0&0&0&0&0&0&0&0&0&0
\\ 0&0&0&0&0&0&0&{\frac {1}{\sqrt {q}}}&0&0&0&0&0&0&0
&0\\ 0&0&0&0&0&0&0&0&0&0&0&\sqrt {q}&0&0&0&0
\\ 0&0&0&0&0&0&0&0&0&0&0&0&0&0&0&{\frac {1}{\sqrt {q
}}}\end{smallmatrix} \right)
\end{equation}
with the inverse
\begin{equation}\label{irmat}
\left(
\begin{smallmatrix}\sqrt{q}&0&0&0&0
&0&0&0&0&0&0&0&0&0&0&0\\
0&0&0&0&\sqrt {q}&0&0&0&0&0&0&0&0&0&0&0\\
0&0&0&0&0&0&0&0&{\frac {1}{\sqrt {q}}}&0&0&0&0&0&0&0
\\ 0&0&0&0&0&0& 0&0&0&0&0&0&{\frac {1}{\sqrt {
q}}}&0&0&0\\0&{\frac {1}{\sqrt {
q}}}&0&0&{-\frac{1-q}{\sqrt {q
}}}&0&0&0&0&0&0&0&0&0&0&0\\0&0&0&0&0&\sqrt {q}&0&0&0&0&0&0&0&0&0&0\\0&0&0&0&0&0&0&0&0&{\frac {1}{\sqrt {
q}}}&0&0&{-\frac{1-q}{q\sqrt {q
}}}&0&0&0
\\ 0&0&0&0&0&0&0&0&0&0&0&0&0
&\sqrt {q}&0&0\\ 0&0&\sqrt {q}&0&0&0&0&0&{-\frac{1-q}{\sqrt {q
}}}&0&0&0&0&0
&0&0\\ 0&0&0&0&0&0&{\frac {1}
{\sqrt {q}}}&0&0&0&0&0&{-\frac{1-q}{q\sqrt {q
}}}&0&0&0\\0&0&0&0&0&0&0&0&0&0&\sqrt {q}&0&0&0&0&0\\ 0&0&0&0&0&0&0&0&0&0&0&
0&0&0&{\frac {1}{\sqrt {q}}}&0
\\0&0&0&{\frac {1}{\sqrt {q}}}&0&0&{-\frac {1-q}{\sqrt{q}}}&0&0&{-\frac{1-q}{\sqrt {q
}}}&0&0&{\frac{(1-q)^2}{q\sqrt {q
}}}&0&0&0
\\ 0&0&0&0&0&0&0&{\frac {1}{\sqrt {q}}}&0&0&0&0&0&{-\frac {1-q}{\sqrt{q}}}&0
&0\\ 0&0&0&0&0&0&0&0&0&0&0&\sqrt {q}&0&0&{-\frac {1-q}{\sqrt{q}}}&0
\\ 0&0&0&0&0&0&0&0&0&0&0&0&0&0&0&\sqrt{q}
\end{smallmatrix} \right)
\end{equation}
with respect to the basis
\[
\{b_1\ot b_1,~b_1\ot b_2,~b_1\ot b_3,~b_1\ot b_4,~\dots,
~b_4\ot b_1,~b_4\ot b_2,~b_4\ot b_3,~b_4\ot b_4\}.
\]

Note that the braiding $R$-matrix above is not symmetric which is different from those basic braiding $R$-matrices
arising from the quantum groups of type $BCD$ over the tensor square of vector representation (in the latter case, they are symmetric, see \cite{HH1}, \cite{HH2}).

When $\ell$ is large enough, we have the following irreducible decomposition of tensor square,
\[\bar{\mathscr{A}}_{(\ell-1)\ep_2,i_0(\ep_1+\ep_2)}^{\ot 2}\cong \bar{\mathscr{A}}_{(\ell-2)\ep_2,\ep_1+\ep_2}\oplus\bar{\mathscr{A}}_{\ep_1+(\ell-1)\ep_2,\ep_2}\oplus
\bar{\mathscr{A}}_{(\ell-1)(\ep_1+\ep_2),\ep_1}\oplus\bar{\mathscr{A}}_{0,0},\]
where $\bar{\mathscr{A}}_{0,0}$ is the one-dimensional trivial representation spanned by
\begin{equation}\label{trep}
v_{0,0}:=
q\, b_1\ot b_4 - b_2\ot b_3 - b_3\ot b_2 + b_4\ot b_1.
\end{equation}
By Schur's lemma, $R_{(\ell-1)\ep_2,i_0(\ep_1+\ep_2)}$ acts by multiplying scalars $q^{-1/2},-q^{1/2},-q^{1/2},q^{3/2}$ on these four simple submodules successively
, with the minimal polynomial
\begin{equation}\label{poly2}
\ms{minpoly}(R)=\lb x-q^{-1/2}\rb\lb x+q^{1/2}\rb\lb x-q^{3/2}\rb.
\end{equation}

Using the spectral decomposition of the braiding matrix together with its inverse, we have
\begin{align}
\nonumber R-R^{-1}&=(q^{-1/2}-q^{1/2})\mi_{\bar{\mathscr{A}}_{(\ell-2)\ep_2,\ep_1+\ep_2}\oplus\bar{\mathscr{A}}_{\ep_1+(\ell-1)\ep_2,\ep_2}\oplus
\bar{\mathscr{A}}_{(\ell-1)(\ep_1+\ep_2),\ep_1}}+(q^{3/2}-q^{-3/2})\mi_{\bar{\mathscr{A}}_{0,0}}\\
\label{skr}& =(q^{-1/2}-q^{1/2})\lb\mi_{\bar{\mathscr{A}}_{(\ell-1)\ep_2,i_0(\ep_1+\ep_2)}^{\ot 2}}-([3]_{q^{1/2}}+1)\mi_{\bar{\mathscr{A}}_{0,0}}\rb\\
&\nonumber
=(q^{-1/2}-q^{1/2})\lb\mi_{\bar{\mathscr{A}}_{(\ell-1)\ep_2,i_0(\ep_1+\ep_2)}^{\ot 2}}-(q^{-1}+2+q)\mi_{\bar{\mathscr{A}}_{0,0}}\rb,
\end{align}
where identity maps are interpreted as the corresponding projections of irreducible submodules. Accordingly, by \eqref{rmat} and \eqref{irmat} the quasi-projection \[(q^{-1}+2+q)\mi_{\bar{\mathscr{A}}_{0,0}}=
\mi_{\bar{\mathscr{A}}_{(\ell-1)\ep_2,i_0(\ep_1+\ep_2)}^{\ot 2}}-(q^{-1/2}-q^{1/2})^{-1}(R-R^{-1})\]
is presented by
\begin{equation}\label{qide}
\left(
\begin{smallmatrix}0&0&0&0&0
&0&0&0&0&0&0&0&0&0&0&0\\
0&0&0&0&0&0&0&0&0&0&0&0&0&0&0&0\\
0&0&0&0&0&0&0&0&0&0&0&0&0&0&0&0
\\ 0&0&0&q&0&0&-q&0&0&-q&0&0&1&0&0&0\\0&0&0&0&0&0&0&0&0&0&0&0&0&0&0&0\\
0&0&0&0&0&0&0&0&0&0&0&0&0&0&0&0\\0&0&0&-1&0&0&1&0&0&1&0&0&{-q^{-1}}&0&0&0
\\ 0&0&0&0&0&0&0&0&0&0&0&0&0
&0&0&0\\ 0&0&0&0&0&0&0&0&0&0&0&0&0&0
&0&0\\ 0&0&0&-1&0&0&1&0&0&1&0&0&{-q^{-1}}&0&0&0\\
0&0&0&0&0&0&0&0&0&0&0&0&0&0&0&0\\ 0&0&0&0&0&0&0&0&0&0&0&
0&0&0&0&0
\\0&0&0&1&0&0&-1&0&0&-1&0&0&q^{-1}&0&0&0
\\ 0&0&0&0&0&0&0&0&0&0&0&0&0&0&0
&0\\ 0&0&0&0&0&0&0&0&0&0&0&0&0&0&0&0
\\ 0&0&0&0&0&0&0&0&0&0&0&0&0&0&0&0
\end{smallmatrix} \right)
\end{equation}
with respect to the basis
\[
\{b_1\ot b_1,~b_1\ot b_2,~b_1\ot b_3,~b_1\ot b_4,~\dots,
~b_4\ot b_1,~b_4\ot b_2,~b_4\ot b_3,~b_4\ot b_4\}.
\]

The cubic minimal polynomial (\ref{poly2}) of $R_{(\ell-1)\ep_2,i_0(\ep_1+\ep_2)}$ reminds us about the famous Birman-Murakami-Wenzl algebra; e.g. see \cite{BW}, \cite{W} or \cite[8.6.4]{KS}.
\begin{definition}\label{bmw}
The Birman-Murakami-Wenzl algebra $\mathscr{B}_r(s,t)$ is an associative $k$-algebra   with 1, whose
 generating set is $\{T_i\,|\,1\leq i\leq r-1\}$ , and satisfies the relations
\[
\begin{array}{l}
\ms{\upshape (a)}\quad (T_i-s)(T_i+s^{-1})(T_i-t^{-1})=0, \quad 1\leq i\leq r-1,\\
\ms{\upshape (b)}\quad  T_iT_{i+1}T_i=T_{i+1}T_iT_{i+1}, \quad 1\leq i\leq r-1,\\
\ms{\upshape (c)}\quad  T_iT_j=T_jT_i,\quad |i-j|>1,\\
\ms{\upshape (d)}\quad  E_iT_j^{\pm1} E_i=t^{\pm1} E_i, \quad 1\leq i\leq r-1, \ j=i\pm 1,\\
\ms{\upshape (e)}\quad  E_iT_i=T_iE_i=t^{-1}E_i, \quad 1\leq i\leq r-1,
\end{array}
\]
where the quasi-idempotent $E_i=1-\om^{-1}\lb T_i-T_i^{-1}\rb,\ 1\leq i\leq k-1$, $\om=s-s^{-1}$.
\end{definition}
In fact,
relation {\sffamily (e)} is due to relation {\sffamily (a)}, and it implies that
\begin{equation}\label{idt1}
E_i^2=E_i\lb1-\om^{-1}\lb T_i-T_i^{-1}\rb\rb=(1+\om^{-1}\te)E_i,~1\leq i\leq r-1\text{ with }\te=t-t^{-1}.
\end{equation}
By relation {\sffamily (d)} and \eqref{idt1}, one can also see that
\begin{equation}\label{idt2}
E_iE_jE_i=E_i\lb1-\om^{-1}\lb T_j-T_j^{-1}\rb\rb E_i=E_i^2-\om^{-1}\te E_i=E_i
\end{equation}
for $1\leq i\leq r-1,j=i\pm 1$. Also, relation {\sffamily (c)} implies that
\begin{equation}\label{idt3}
E_iE_j=E_jE_i,\,|i-j|>1.
\end{equation}
Therefore, $\{E_i\,|\,1\leq i\leq r-1\}$ generates a subalgebra isomorphic to the Temperley-Lieb algebra $\mathscr{T}_r(1+\om^{-1}\te)$.

Let us consider the endomorphism algebra \[\mathscr{E}_r=\ms{End}_{D(\bar{\mathscr{A}}_q(2))}\lb\bar{\mathscr{A}}_{(\ell-1)\ep_2,i_0(\ep_1+\ep_2)}^{\ot r}\rb,\]
and the specialization
\[\mathscr{B}_r(q):=\mathscr{B}_r\lb q^{-1/2},q^{-3/2}\rb.\]
Take
\[R_i:=\mi_{\bar{\mathscr{A}}_{(\ell-1)\ep_2,i_0(\ep_1+\ep_2)}}^{\ot(i-1)}\ot R_{(\ell-1)\ep_2,i_0(\ep_1+\ep_2)}\ot\mbox{Id}_{\bar{\mathscr{A}}_{(\ell-1)\ep_2,i_0(\ep_1+\ep_2)}}^{\ot(r-i-1)},~1\leq i\leq r-1.\]
Then we obtain that
\begin{theorem}\label{end}
There exists an algebra homomorphism
\begin{equation}
\Phi :\mathscr{B}_r(q)\longrightarrow \mathscr{E}_r, \quad T_i\longmapsto R_i, \ 1\leq i\leq r-1,
\end{equation}
for $r\ge 2$.
Especially, \[e_i:=\Phi(E_i)=\mbox{Id}_{\bar{\mathscr{A}}_{(\ell-1)\ep_2,i_0(\ep_1+\ep_2)}}^{\ot r}-\lb q^{-1/2}-q^{1/2}\rb^{-1}\lb R_i-R_i^{-1}\rb.
\]
\end{theorem}
\begin{proof}
Indeed, as $R$ is a braiding $R$-matrix satisfying \eqref{poly2} and \eqref{skr}, we know that $\Phi$ can factor through relations {\sffamily (a)}--{\sffamily (c)} and {\sffamily (e)}. In particular, the $2$-fold tensor $v_{0,0}$ in \eqref{trep} is an eigenvector of $R$ with eigenvalue $q^{3/2}$. On the other hand, using \eqref{rmat}, \eqref{irmat} and \eqref{qide} we check that
\begin{align*}
&e_1(R_2)^{\pm1}(v_{0,0}\ot b_i)=q^{\mp3/2}(v_{0,0}\ot b_i),\\
&e_2(R_1)^{\pm1}(b_i\ot v_{0,0})=q^{\mp3/2}(b_i\ot v_{0,0})
\end{align*}
for $1\leq i\leq 4$.
That means $\Phi$ also factors through relation {\sffamily (d)}, so it is well-defined.
\end{proof}

\begin{proposition}\label{subalg}
The endomorphism algebra $\mathscr{E}_r$ has a quotient of $\mathscr{B}_r(q)$ as its proper subalgebra.
\end{proposition}
\begin{proof}
According to \cite[Theorem~2.7]{MO}, for a Hopf algebra $H$ with a Hopf 2-cocycle $\si$,  ${_H\mathcal {Y}\mathcal {D}^H}$ is equivalent to  ${_{H^\si}\mathcal {Y}\mathcal {D}^{H^\si}}$ as braided monoidal categories. Hence, by Theorem~\ref{twist} $D(\bar{\mathscr{A}}_q(2))$-mod is equivalent to $D(\bar{\mathscr{A}}_q(1))^{\ot 2}$-mod, same as $u_q(sl_2)^{\ot2}$-mod, and $\bar{\mathscr{A}}_{(\ell-1)\ep_2,i_0(\ep_1+\ep_2)}$ actually corresponds to the tensor square of fundamental representations, $V_q(1)^{\ot2}$.
As a result, $\ms{dim}\lb\mathscr{E}_r\rb=\lb\tfrac{1}{r+1}{2r\choose r}\rb^2$, which is equal to \[\ms{dim}\lb\ms{End}_{u_q(sl_2)^{\ot2}}\lb \lb V_q(1)^{\ot2}\rb^{\ot r}\rb\rb=\lb\ms{dim}\lb\ms{End}_{u_q(sl_2)}\lb V_q(1)^{\ot r}\rb\rb\rb^2,\]
namely the dimension of tensor square of the Temperley-Lieb algebra.
As $\ms{dim}\lb\mathscr{B}_r(q)\rb=(2r-1)!!<\ms{dim}\lb\mathscr{E}_r\rb, r\ge 2$, $\Phi$ is not surjective and $\ms{Im}(\Phi)\subsetneq \mathscr{E}_r$.
\end{proof}

\begin{remark}
As we have seen, our endomorphism algebra $\mathscr{E}_r$ on $\bar{\mathscr{A}}_{(\ell-1)\ep_2,i_0(\ep_1+\ep_2)}^{\ot r}$ is distinct to the BMW algebra arising from the fundamental representation of the BCD type quantum group. But in Theorem \ref{sqj} we will see that it induces the same skein relation as that of the so-called Dubrovnik invariant.
\end{remark}

\begin{remark}
By computer program executing the algorithm of Kauffmann and Radford, we list evaluations of $\overline{\mb{Inv}}_{D(\bar{\mathscr{A}}),\bar{\mathscr{A}}_{0,\ep_2}}$ on the Hopf link and some small knots as follows with normalization $\overline{\mb{Inv}}(\mb{O})=1$. The reader can compare and refer to \cite[Appendix A, D]{Cro} for the information of these small knots and their Jones polynomials.

\smallskip
 \centerline{
\begin{tabular}{|c|c|}
\hline $\mb{K}$ &
$\overline{\mb{Inv}}_{D(\bar{\mathscr{A}}),\bar{\mathscr{A}}_{0,\ep_2}}$\\
\hline
 $2_1^2(L)$  & $q^{-5}+2q^{-3}+q^{-1}$\\
\hline
 $2_1^2(R)$  & $q+2q^3+q^5$\\
\hline
 $3_1(L)$  & $q^{-8}-2q^{-7}+q^{-6}-2q^{-5}+2q^{-4}+q^{-2}$\\
\hline
 $3_1(R)$  & $q^2+2q^4-2q^5+q^6-2q^7+q^8$\\
\hline
 $4_1(A)$  & $q^{-4}-2q^{-3}+3q^{-2}-4q^{-1}+5-4q+3q^2-2q^3+q^4$\\
\hline
 $5_1(L)$ & $q^{-14}-2q^{-13}+3q^{-12}-4q^{-11}+3q^{-10}-4q^{-9}+3q^{-8}-2q^{-7}+2q^{-6}+q^{-4}$\\
\hline
 $5_1(R)$ &  $q^4+2q^6-2q^7+3q^8-4q^9+3q^{10}-4q^{11}+3q^{12}-2q^{13}+q^{14}$\\
\hline
 $5_2(L)$ &  $q^{-12}-2q^{-11}+3q^{-10}-6q^{-9}+7q^{-8}-8q^{-7}+8q^{-6}-6q^{-5}+5q^{-4}-2q^{-3}+q^{-2}$\\
\hline
 $5_2(R)$ &  $q^2-2q^3+5q^4-6q^5+8q^6-8q^7+7q^8-6q^9+3q^{10}-2q^{11}+q^{12}$\\
\hline
\end{tabular}
}

\end{remark}

The observation at the list above leads us to conclude that
\begin{theorem}\label{sqj}
The following relation of ambient isotopy invariants holds,
\begin{equation}\label{squ}
\overline{\mb{Inv}}_{D(\bar{\mathscr{A}}_2(q)),\bar{\mathscr{A}}_{(\ell-1)\ep_2,i_0(\ep_1+\ep_2)}}
=\lb\overline{\mb{Inv}}_{D(\bar{\mathscr{A}}_1(q)),\bar{\mathscr{A}}_{(i_0-1)\ep_1,i_0\ep_1}}\rb^2.
\end{equation}
\end{theorem}
\begin{proof} We will give a skein-theoretic argument in generality for the computational result above.
In fact, combining relation (5.8) with Def. \ref{bmw} \& (5.9), we arrive at the skein relation of Dubrovnik type if set the quasi-idempotent $E_i=(q^{-1}{+}2{+}q)\,\mi_{\bar{\mathscr{A}}_{0,0}}$ in \eqref{skr} .

Recall the Dubrovnik skein relation gives rise to the {\itshape framed Dubrovnik polynomial}
$\widetilde{\mbox{Dubrovnik}}(L)(a,z)$, which is a regular isotopy invariant of unoriented links defined by
\[\xy
0;/r.1pc/:
(-6,6)*{};(-1.5,1.5)*{}**\dir{-};
(6,-6)*{};(1.5,-1.5)*{}**\dir{-};
(-6,-6)*{};(6,6)*{}**\dir{-};
\endxy-
\xy
0;/r.1pc/:
(-6,-6)*{};(-1.5,-1.5)*{}**\dir{-};
(6,6)*{};(1.5,1.5)*{}**\dir{-};
(-6,6)*{};(6,-6)*{}**\dir{-};
\endxy=
z\lb~\xy
0;/r.1pc/:
(-6,-6)*{};(-6,6)*{}**\crv{(0,0)};
(6,6)*{};(6,-6)*{}**\crv{(0,0)};
\endxy-
\xy
0;/r.1pc/:
(-6,-6)*{};(6,-6)*{}**\crv{(0,0)};
(6,6)*{};(-6,6)*{}**\crv{(0,0)};
\endxy
~\rb,
\]
\[\xy
0;/r.1pc/:
(-6,6)*{};(-1.5,1.5)*{}**\dir{-};
(5,-5)*{};(1.5,-1.5)*{}**\dir{-};
(-6,-6)*{};(5,5)*{}**\dir{-};
(5,5)*{};(5,-5)*{}**\crv{(10,10)&(10,-10)};
\endxy=a~\xy
0;/r.1pc/:
(-6,6)*{};(-6,-6)*{}**\crv{(0,0)};
\endxy~,~
\xy
0;/r.1pc/:
(-6,-6)*{};(-1.5,-1.5)*{}**\dir{-};
(5,5)*{};(1.5,1.5)*{}**\dir{-};
(-6,6)*{};(5,-5)*{}**\dir{-};
(5,5)*{};(5,-5)*{}**\crv{(10,10)&(10,-10)};
\endxy=a^{-1}~\xy
0;/r.1pc/:
(-6,6)*{};(-6,-6)*{}**\crv{(0,0)};
\endxy~,
\]
and the value of disjoint union of $m$ unknots is $\lb\tfrac{a-a^{-1}}{z}+1\rb^{m-1}$.
More explicitly, the regular invariant $\mb{Inv}_{D(\bar{\mathscr{A}}_2(q)),\bar{\mathscr{A}}_{(\ell-1)\ep_2,i_0(\ep_1+\ep_2)}}$ is a
specialization of $\widetilde{\mbox{Dubrovnik}}(L)(a,z)$ as follow,
\[\mb{Inv}_{D(\bar{\mathscr{A}}_2(q)),\bar{\mathscr{A}}_{(\ell-1)\ep_2,i_0(\ep_1+\ep_2)}}(L)
=\widetilde{\mbox{Dubrovnik}}(L)(q^{-3/2},q^{-1/2}-q^{1/2}).\]

By writhe-normalization in Prop.~\ref{wnor}, $\overline{\mb{Inv}}_{D(\bar{\mathscr{A}}_2(q)),\bar{\mathscr{A}}_{(\ell-1)\ep_2,i_0(\ep_1+\ep_2)}}$ is a specialization of the {\itshape unframed Dubrovnik polynomial}, an ambient isotopy invariant of oriented links, denoted $\mbox{Dubrovnik}(L)(a,z)$. That is,
\[\mbox{Dubrovnik}(L)(a,z)=a^{-\scriptsize\ms{Wr}(L)}\widetilde{\mbox{Dubrovnik}}(L)(a,z),\]
and in our case,
\begin{align*}
\overline{\mb{Inv}}_{D(\bar{\mathscr{A}}_2(q)),\bar{\mathscr{A}}_{(\ell-1)\ep_2,i_0(\ep_1+\ep_2)}}(L)
&=(q^{3/2})^{\scriptsize\ms{Wr}(L)}\mb{Inv}_{D(\bar{\mathscr{A}}_2(q)),\bar{\mathscr{A}}_{(\ell-1)\ep_2,i_0(\ep_1+\ep_2)}}(L)\\
&=\mbox{Dubrovnik}(L)(q^{-3/2},q^{-1/2}-q^{1/2}).
\end{align*}

On the other hand, we have seen that $\overline{\mb{Inv}}_{D(\bar{\mathscr{A}}_1(q)),\bar{\mathscr{A}}_{(i_0-1)\ep_1,i_0\ep_1}}$ recovers the Jones polynomial $V_L$ in Example~\ref{ex1}.
Consequently, relation \eqref{squ} is due to the following result \cite[Prop. 16.6]{Lic},
\[
\begin{split}
(V_L(t))^2&=\mbox{Dubrovnik}(L)\lb t^{-3/2},t^{-1/2}-t^{1/2}\rb\\
&=(-1)^{\sharp L-1}\mbox{Kauffman}(L)\lb\sqrt{-1}t^{-3/2},\sqrt{-1}\lb t^{1/2}-t^{-1/2}\rb\rb,
\end{split}
\]
where $\mbox{Kauffman}(L)(a,z)=a^{-\scriptsize\ms{Wr}(L)}\widetilde{\mbox{Kauffman}}(L)(a,z)$, and $\widetilde{\mbox{Kauffman}}(L)(a,z)$ is the {\itshape framed Kauffman polynomial}, defined by
\[\xy
0;/r.1pc/:
(-6,6)*{};(-1.5,1.5)*{}**\dir{-};
(6,-6)*{};(1.5,-1.5)*{}**\dir{-};
(-6,-6)*{};(6,6)*{}**\dir{-};
\endxy+
\xy
0;/r.1pc/:
(-6,-6)*{};(-1.5,-1.5)*{}**\dir{-};
(6,6)*{};(1.5,1.5)*{}**\dir{-};
(-6,6)*{};(6,-6)*{}**\dir{-};
\endxy=
z\lb~\xy
0;/r.1pc/:
(-6,-6)*{};(-6,6)*{}**\crv{(0,0)};
(6,6)*{};(6,-6)*{}**\crv{(0,0)};
\endxy+
\xy
0;/r.1pc/:
(-6,-6)*{};(6,-6)*{}**\crv{(0,0)};
(6,6)*{};(-6,6)*{}**\crv{(0,0)};
\endxy
~\rb,
\]
\[\xy
0;/r.1pc/:
(-6,6)*{};(-1.5,1.5)*{}**\dir{-};
(5,-5)*{};(1.5,-1.5)*{}**\dir{-};
(-6,-6)*{};(5,5)*{}**\dir{-};
(5,5)*{};(5,-5)*{}**\crv{(10,10)&(10,-10)};
\endxy=a~\xy
0;/r.1pc/:
(-6,6)*{};(-6,-6)*{}**\crv{(0,0)};
\endxy~,~
\xy
0;/r.1pc/:
(-6,-6)*{};(-1.5,-1.5)*{}**\dir{-};
(5,5)*{};(1.5,1.5)*{}**\dir{-};
(-6,6)*{};(5,-5)*{}**\dir{-};
(5,5)*{};(5,-5)*{}**\crv{(10,10)&(10,-10)};
\endxy=a^{-1}~\xy
0;/r.1pc/:
(-6,6)*{};(-6,-6)*{}**\crv{(0,0)};
\endxy~,
\]
and the value of disjoint union of $m$ unknots is $\lb\tfrac{a+a^{-1}}{z}-1\rb^{m-1}$.
%
\end{proof}

\medskip
Accordingly, when take the unique $2^n$ dimensional self-dual irreducible representation $(\bar{\mathscr{A}}_{\al_0,\be_0},\pi)$ of $D\lb\bar{\mathscr{A}}_q(n)\rb$ with $\ka(\al_0,\be_0)=\ep_1+\cdots+\ep_n$,
one can obtain the following results:
\begin{corollary}\label{nrank}
(a) The braiding $R$-matrix $R_{\al_0,\be_0}$ has
the minimal polynomial
\begin{equation}
\ms{minpoly}(R_{\al_0,\be_0})=\prod\limits_{j=0}^n\lb x-(-1)^jq^{\tfrac{-n+4j}{4}}\rb.
\end{equation}
On the other hand, the ribbon element $v$ acts by multiplying $q^{3n/4}$.

\smallskip
(b) The knot invariant of $D(\bar{\mathscr{A}}_n(q))$ just has the composite relation
\[\overline{\mb{Inv}}_{D(\bar{\mathscr{A}}_n(q)),\bar{\mathscr{A}}_{\al_0,\be_0}}
=\lb\overline{\mb{Inv}}_{D(\bar{\mathscr{A}}_1(q)),\bar{\mathscr{A}}_{(i_0-1)\ep_1,i_0\ep_1}}\rb^n.\]
\end{corollary}


\bibliographystyle{amsalpha}


\begin{thebibliography}{A}
\bibitem[BW]{BW} J. Birman and H. Wenzl, \textit{Braids, link polynomials and a new algebra}, Trans. AMS
\textbf{313} (1989), 249--273.

\bibitem[CL]{CL} Q. Chen and K. Liu, \textit{New structure for orthogonal quantum group invariants}, Proc. Amer. Math. Soc. 143 (8) (2015), 3645--3657.

\bibitem[CR]{CR} Q. Chen and N. Reshetikhin, \textit{Recursion formulas for HOMFLY and Kauffman invariants}, J. Knot Theory Ramifications 23 (5) (2014), 1450024, 23 pp.

\bibitem[CY]{CY} Q. Chen and T. Yang,  \textit{Volume conjectures for the Reshetikhin-Turaev and the Turaev-Viro invariants}, Quantum Topol. 9 (3) (2018), 419--460.

\bibitem[Cro]{Cro} P.~R. Cromwell, \textit{Knots and Links}, Cambridge University Press, 2004.

\bibitem[Dri]{Dri} V.~G. Drinfeld, \textit{Quantum groups}, Proceedings of ICM, vol. {\bf 1, 2} (Berkeley, Calif., 1986), 798--820, AMS, Providence, RI, 1987.


\bibitem[HH1]{HH1} H. Hu and N. Hu, \textit{Double-bosonization and Majid's conjecture, (I): Rank-inductions of $ABCD$}, J. Math. Phys. \textbf{56} (11) (2015), 16 pp.

\bibitem[HH2]{HH2} ---, \textit{Double-bosonization and Majid's conjecture, (IV): Type-crossings from $A$ to $BCD$}, Sci. China Math. \textbf{59} (6) (2016), 1061---1080.


\bibitem[Hu1]{Hu1} N. Hu, \textit{Quantum divided power algebra, $q$-derivatives,
and some new quantum groups}, J. Algebra, \textbf{232} (2000), 507--540.

\bibitem[Hu2]{Hu2} ---, \textit{Quantum group structure
associated to the quantum affine space}, (Pr\'epublication de IRMA, Strasbourg, Preprint 2001, No. 26). Algebra Colloq., \textbf{11} (4) (2004), 483--492.

\bibitem[Ka]{Ka} C. Kassel, \textit{Quantum Groups}, Springer-Verlag, Graduate Texts in Mathematics, \textbf{155}, 1995.

\bibitem[Kau1]{Kau} L.~H. Kauffman, \textit{An invariant of regular isotopy}, Trans. Amer. Math. Soc., \textbf{318} (2) (1990), 317--371.

\bibitem[Kau2]{Kau1} ---, \textit{Knots and Physics}, World Scientific, Series of Knots and Everything, \textbf{1}, 1991.

\bibitem[KR1]{KR} L.~H. Kauffman and D.~E. Radford, \textit{Oriented quantum algebras and invariants of knots and links},
J. Algebra, \textbf{246} (2001), 253--291.

\bibitem[KR2]{KR1} ---, \textit{A necessary and sufficient condition for a finite-dimensional
Drinfel'd double to be a ribbon Hopf algebra}, J. Algebra, \textbf{159} (1993), 98--114.


\bibitem[Kh]{Kh} M. Khovanov, \textit{A categorification of the Jones polynomial}, Duke Math. J., \textbf{101} (3) (2000), 359--426.


\bibitem[KS]{KS} A. Klimyk and K. Schm\"udgen, \textit{Quantum groups and their representations}, Texts and Monographs in Physics. Springer-Verlag, Berlin, 1997.

\bibitem[LQR]{LQR} A. Lauda, H. Queffelec, D. Rose, \textit{Khovanov homology is a skew Howe 2-representation of categorified quantum $\mathfrak{s}\mathfrak{l}_m$}, Algebr. Geom. Topol. \textbf{15} (5) (2015), 2517--2608.

\bibitem[LH]{LH} Y. Li and N. Hu, The Green rings of the 2-rank Taft algebra and its two relatives twisted. J. Algebra, \textbf{410} (2014), 1--35.

\bibitem[Lic]{Lic} W.~B. R. Lickorish, \textit{An Introduction to Knot Theory}, Springer-Verlag, Graduate Texts in Mathematics, \textbf{175},
1997.

\bibitem[LP]{LP} K. Liu and P. Peng, \textit{Proof of the Labastida-Marino-Ooguri-Vafa conjecture}, J. Diff. Geom. \textbf{85} (3) (2010), 479--525.

\bibitem[Ma]{Ma} Yu.~I. Manin, \textit{Quantum Groups and Non-commutative Geometry},
Universit\"e de Montr\"eal, 1988.

\bibitem[Mon]{Mon} S. Montgomery, \textit{Hopf Algebras and Their Actions on Rings}, Amer. Math. Soc., Regional Conf. Ser. in Math., \textbf{82}, 1993.

\bibitem[MO]{MO} S. Majid and R. Oeckl, \textit{Twisting of quantum differentials and the Planck scale Hopf algebra}, Comm. Math. Phys. \textbf{205} (1999), 617--655.

\bibitem[Mo]{Mo} P. Morandi, \textit{Field and Galois Theory}, Springer-Verlag, Graduate Texts in Mathematics, \textbf{167}, 1996.

\bibitem[MPS]{MPS} S. Morrison, E. Peters and N. Snyder, \textit{Knot polynomial identities and quantum group coincidences, from the $D_{2n}$ planar algebra}, Quantum Topol.  \textbf{2} (2)  (2011),  101--156.

\bibitem[Oht]{Oht} T. Ohtsuki, \textit{Quantum Invariants, A Study of Knots, $3$-Manifolds, and Their Sets}, World Scientific, Series of Knots and Everything,
\textbf{29}, 2002.

\bibitem[Rad1]{Rad} D.~E. Radford, \textit{On oriented quantum algebras derived from representations of the quantum double of a finite-dimensional Hopf algebra},
J. Algebra, \textbf{270} (2003), 670--695.

\bibitem[Rad2]{Rad1}---, \textit{Minimal quasitriangular Hopf algebras},
J. Algebra, \textbf{157} (1993), 285--315.

\bibitem[Rad3]{Rad2}---, \textit{On quantum algebras and coalgebras, oriented
quantum algebras and coalgebras, invariants of
1-1 tangles, knots and links}, New Directions in Hopf Algebras, MSRI Publications, \textbf{43}, 2002.

\bibitem[RW]{RW} D.~E. Radford and S. Westreich, \textit{Trace-like function on the Drinfel'd double of Taft algebras},
J. Algebra, \textbf{301} (2006), 1--34.

\bibitem[RT1]{RT1} N. Reshetinkhin and V.~G. Turaev, \textit{Ribbon graphs and their invariants derived
from quantum groups}, Comm. Math. Phys., \textbf{127} (1990), 1--26.

\bibitem[RT2]{RT2} ---, \textit{Invariants of $3$-manifolds via link polynomials and quantum groups},
Invent. Math., \textbf{103} (1991), 547--597.

\bibitem[Ro]{Ro} M. Rosso, \textit{Quantum groups at a root of $1$ and tangle invariants},
Internat. J. Modern Phys. B, \textbf{7} (1993), 3715--3726.

\bibitem[Ta]{Ta} E.~J. Taft, \textit{The order of the antipode of finite-dimensional Hopf algebra},
Proc. Nat. Acad. Sci. U.S.A., \textbf{68} (1971), 2631--2633.

\bibitem[Tub]{Tub}
D. Tubbenhauer, \textit{ $\mathfrak{gl}_n$-webs, categorification and Khovanov-Rozansky homologies}. J. Knot Theory Ram. {\bf 29} (11) (2020), 2050074, 96 pp.

\bibitem[Tu1]{Tu} V.~G. Turaev, \textit{The Yang-Baxter equation and invariants of links},
Invent. Math., \textbf{92} (1988), 527--553.

\bibitem[Tu2]{Tu1} ---, \textit{Quantum invariants of knots and $3$-manifolds}, De Gruyter Studies in Mathematics, vol.{\bf 18}, Walter de Gruyter \& Co. Berlin, 1994.

\bibitem[WW]{WW} B. Webster, G. Williamson, \textit{A geometric construction of colored HOMFLYPT homology}, Geom. Topol. {\bf 21} (5) (2017), 2557--2600.

\bibitem[We]{W} H. Wenzl, \textit{Quantum groups and subfactors of type $B$, $C$, and $D$}, Comm. Math. Phys. \textbf{133} (1990), 383--432.

\bibitem[Wi]{Wi} E. Witten, \textit{Quantum field theory and the Jones polynomial}, Comm. Math. Phys.  \textbf{121} (1989), 351--399.

\end{thebibliography}

\end{document}